\documentclass[11pt]{article}
\usepackage{amsmath,amscd}
\usepackage{amssymb,latexsym,amsthm}
\usepackage{color}
\usepackage{mathpazo}
\usepackage[utf8]{inputenc}
\usepackage[T1]{fontenc}
\usepackage{csquotes}
\usepackage{multirow}
\usepackage{lscape}
\usepackage{enumerate}
\usepackage{longtable}
\usepackage{array}
\usepackage{fancyhdr}
\usepackage{amsfonts}
\usepackage{hyperref}
\usepackage{url}
\hypersetup{
  colorlinks=true,
  linkcolor=blue,
  citecolor=black,
  urlcolor=cyan
}
\newtheorem{theorem}{Theorem}[section]
\newtheorem{proposition}[theorem]{Proposition}
\newtheorem{lemma}[theorem]{Lemma}
\newtheorem{corollary}[theorem]{Corollary}

\theoremstyle{definition}
\newtheorem{remark}[theorem]{Remark}

\newtheorem{example}[theorem]{Example}

\numberwithin{equation}{section}

\newcommand{\pqnum}[1]{\left[#1\right]_{p,q}}
\newtheorem*{theoremA}{Theorem A}
\newtheorem*{theoremB}{Theorem B}
\DeclareMathOperator{\Gal}{Gal}

\DeclareMathOperator{\GKdim}{GKdim}

\DeclareMathOperator{\Oz}{Oz}
\DeclareMathOperator{\Aut}{Aut}
\DeclareMathOperator{\ord}{ord}
\DeclareMathOperator{\rk}{rk}
\DeclareMathOperator{\PIdeg}{PIdeg}
\DeclareMathOperator{\gr}{gr}

\newcommand{\K}{\Bbbk}

\newcommand{\Z}{\mathbb{Z}}

\newcommand{\C}{\mathbb{C}}
\makeatletter
\def\@roman#1{\romannumeral #1}
\makeatother

\title{\textbf{Ozone groups and centers of two-parameter quantum Heisenberg algebras}}
\author{James G\'omez, H\'ector Su\'arez and Helbert Venegas}
\date{}

\begin{document}
\maketitle

\begin{abstract}
\noindent
Let $p,q$ be roots of unity with $pq\neq1$, and put
$\ell=|\langle p,q\rangle|$ and $t=\ord(pq)$. Up to isomorphism, the
two-parameter quantum Heisenberg algebras $\mathcal H_{p,q}$ are exactly the
noetherian PI graded down-up algebras, so everything here holds for that whole
class. The ozone group is $\mu_\ell\times\mu_{\ell/t}$, generated by
conjugation by three normal elements, and the center is recovered from the
fixed ring of the ozone group as the invariants of a cyclic group of order
$t$. We give a presentation of the center as a hypersurface over a cyclic
quotient surface singularity. It is regular exactly when $1$ is a
characteristic root, which disproves a conjecture of Chan, Gaddis, Won and
Zhang, and Gorenstein exactly when $p^{m/h}=q^{n/h}$, where $m=\ord(p)$,
$n=\ord(q)$ and $h=\gcd(m,n)$; for $p=\zeta_5$, $q=\zeta_5^{2}$ it is not. We
also decide when the fixed ring of the ozone group is Gorenstein and when
$\mathcal H_{p,q}$ is Calabi--Yau, and show that the algebras with regular
center have non-affine automorphisms and nonzero locally nilpotent
derivations.

\bigskip

\end{abstract}


\section{Introduction}

Kirkman and Small \cite{ks} introduced the quantum Heisenberg algebra
\begin{align*}
\mathcal H_q&=\K\langle x,y,z\rangle/(zx-qxz,\ yz-qzy,\ xy-qyx-z),
\end{align*}
a deformation of the enveloping algebra of the three-dimensional Heisenberg
Lie algebra. It is a noetherian AS regular domain of global dimension three,
and a central quotient of it is Hayashi's $q$-analogue of the Weyl algebra
\cite{hayashi}. Gaddis \cite{Gaddis2016} added a second parameter,
\begin{align}
\mathcal H_{p,q}&=\K\langle x,y,z\rangle/(zx-p^{-1}xz,\ zy-pyz,\ yx-qxy-z),
& p,q&\in\K^{\times},
\label{eq:rel}
\end{align}
an iterated Ore extension which, for $pq\neq1$, is PI exactly when $p$ and
$q$ are roots of unity \cite[Propositions~3.4 and~3.11]{Gaddis2016}. The
condition $pq\neq1$ is missing there, and it cannot be dropped
(Section~\ref{sec:pq1}).

Since $z=yx-qxy$, the algebra is generated by $x$ and $y$, and eliminating
$z$ \cite[(3.7)]{Gaddis2016} identifies it with a graded down-up algebra:
\begin{align*}
\mathcal H_{p,q}&\cong A(\alpha,\beta)
=\K\langle x,y\rangle/(y^{2}x-\alpha yxy-\beta xy^{2},\
yx^{2}-\alpha xyx-\beta x^{2}y), &
\alpha&=p^{-1}+q,\; \beta=-p^{-1}q .
\end{align*}
The characteristic roots, the roots of $X^{2}-\alpha X-\beta$, are
$\omega_1=q$ and $\omega_2=p^{-1}$, and the two normal elements
$\Omega_i=yx-\omega_i xy$ are $z$ and $\theta:=yx-p^{-1}xy$. Up to
isomorphism, the $\mathcal H_{p,q}$ with $p,q$ roots of unity and $pq\neq1$
are exactly the noetherian PI graded down-up algebras
(Proposition~\ref{prop:Hpq-downup}). So this paper is about that whole class,
and we switch between the two presentations as convenient.

The ozone group of an algebra $A$ with center $Z$ is
$\Oz(A)=\Aut_{Z\text{-alg}}(A)$, the automorphisms fixing $Z$ pointwise.
Chan, Gaddis, Won and Zhang introduced it to study centers of PI skew
polynomial rings \cite{ChanGaddisWonZhang2024OzoneGroupsCenters}. In
\cite{ChanGaddisWonZhang2025} they showed that for a noetherian PI AS regular
algebra $\Oz(A)$ is finite of order dividing $\rk_ZA$, that a trivial ozone
group forces the Calabi--Yau property, and that among algebras generated in
degree one the skew polynomial rings are the ones with abelian ozone group of
order $\rk_ZA$. Down-up algebras were proposed as a test case
\cite[Section~6]{ChanGaddisWonZhang2024OzoneGroupsCenters}, and part of the
work is already there. Ozone automorphisms of noetherian PI down-up algebras
are diagonal on $x$ and $y$, so the group is abelian
\cite[Corollary~4.7]{ChanGaddisWonZhang2025}; Liu, Wu and Zhu later proved
that the ozone group of every noetherian PI AS regular algebra is abelian
\cite[Theorem~3.4]{LiuWuZhu2026OzoneGroups}. For $A(0,1)$ and $A(0,-1)$ the
ozone group, the center and the rank are computed in
\cite[Example~4.8]{ChanGaddisWonZhang2025}. For these two algebras
$A^{\Oz(A)}$ is commutative and the center is its fixed ring under a Galois
group of order two \cite[Examples~5.4 and~5.5]{ChanGaddisWonZhang2025}; this
is how the center of $A(0,-1)$ is seen to be Gorenstein. In general, the
authors note, no presentation of the center was known, nor whether it is
always Gorenstein \cite[Section~5.1]{ChanGaddisWonZhang2025}. Generators had
been found by Zhao \cite{Zhao1999}. (The quantum Heisenberg algebras of
\cite[Section~1.1]{ChanGaddisWonZhang2025} have $z^{2}$ instead of $z$ in the
last relation. They form a different family.)

Here we carry out the computation for the whole class. Throughout, $p,q$ are
roots of unity with $pq\neq1$ and
\begin{align*}
m&=\ord(p), & n&=\ord(q), & \ell&=\operatorname{lcm}(m,n), &
h&=\gcd(m,n), & t&=\ord(pq).
\end{align*}
Then $\ell=\PIdeg\mathcal H_{p,q}$ \cite[Theorem~3.1]{BeraMandalNandy2024},
$t\mid\ell$, and in terms of the roots $t=\ord(\omega_1\omega_2^{-1})$.

\begin{theoremA}
Let $A=A(\alpha,\beta)\cong\mathcal H_{p,q}$ be a noetherian PI graded
down-up algebra with center $Z$.
\begin{enumerate}[\rm (i)]
\item $\Oz(A)\cong\mu_\ell\times\mu_{\ell/t}$. The first factor is generated
by conjugation by $\Omega_1$ and $\Omega_2$, the second by conjugation by the
normal element $x^{t}$ {\rm(Corollaries~\ref{ozono2H_pq}
and~\ref{cor:generadores-ozono})}.
\item $A^{\Oz(A)}$ is commutative, $Z\subsetneq A^{\Oz(A)}\subsetneq A$,
$\rk_ZA^{\Oz(A)}=t$ and $\rk_{A^{\Oz(A)}}A=\ell^{2}/t$
{\rm(Proposition~\ref{prop:fijo}, Corollary~\ref{cor:noskew})}.
\item $\Gal(A^{\Oz(A)}/Z)$ is cyclic of order $t$, with fixed ring $Z$, and
none of its nontrivial elements extends to an automorphism of $A$
{\rm(Proposition~\ref{prop:Ct})}.
\item $A^{\Oz(A)}$ is a complete intersection if and only if it is
Gorenstein, if and only if $\ell/t\in\{1,2\}$
{\rm(Corollary~\ref{cor:fijo-ci})}.
\item $A$ is Calabi--Yau if and only if $\omega_1\omega_2=1$, that is,
$\beta=-1$ {\rm(Corollary~\ref{calabiyau})}.
\end{enumerate}
\end{theoremA}

\begin{theoremB}
Let $A$ be as in Theorem~A, and put $n_i=\ord(\omega_i)$ and
$h=\gcd(n_1,n_2)$.
\begin{enumerate}[\rm (i)]
\item $Z\cong R[U,V]/(UV-F)$, with $U\mapsto x^{\ell}$, $V\mapsto y^{\ell}$,
where $R\subseteq\K[\Omega_1,\Omega_2]$ is the invariant ring of a diagonal
action of $\mu_\ell$ and $F$ is a nonzero multiple of
$(\Omega_2^{t}-\Omega_1^{t})^{\ell/t}$. So $Z$ is a hypersurface over a
cyclic quotient surface singularity, and a Cohen--Macaulay domain of
dimension three {\rm(Proposition~\ref{prop:presentacion})}.
\item $Z$ is regular if and only if $1\in\{\omega_1,\omega_2\}$, that is,
$\alpha+\beta=1$ {\rm(Corollary~\ref{cor:regular-downup})}.
\item $Z$ is a complete intersection if and only if it is Gorenstein, if and
only if $\omega_1^{\,n_1/h}\omega_2^{\,n_2/h}=1$
{\rm(Proposition~\ref{prop:descenso}, Corollary~\ref{cor:ci})}.
\end{enumerate}
\end{theoremB}

We prove Theorem~A(i) twice. The first proof reads the group off the bigraded
support of the center. The second never looks at the center: it classifies
the homogeneous normal elements of $\mathcal H_{p,q}$
(Theorem~\ref{thm:normales}). We prefer the second, because it shows where
the factor $\mu_{\ell/t}$ comes from. The element $x^{t}$ is normal because
$(pq)^{t}=1$ kills the correction term in
$yx^{t}=q^{t}x^{t}y+[t]_{p,q}x^{t-1}z$ (see \eqref{eq:k1}), and that is all
there is to it. For
$(\alpha,\beta)=(0,\pm1)$ we recover \cite[Example~4.8]{ChanGaddisWonZhang2025},
generators included. The same classification makes explicit the unique
factorization of $\mathcal H_{p,q}$ known from
\cite{LiuWu2026UFR,LaunoisLopes2013}: the central class group is the ozone
group, generated by the classes of $z$, $\theta$ and $x^{t}$
(Remark~\ref{rem:UFR}).

By Theorem~A(ii)--(iii), the iterated fixed ring of
\cite[Section~5.1]{ChanGaddisWonZhang2025} always stops after two steps in
this class, and the Galois group of the second step is cyclic of order $t$;
Examples~5.4 and~5.5 there are the case $t=2$. We doubt this is the way to
the Gorenstein property, though. By Theorem~A(iv) and Theorem~B(iii) the fixed ring of the ozone group and the center are Gorenstein independently of each other,
and all four combinations occur (Remark~\ref{rem:gor-independientes}).
Homological determinants tell the same story. The ozone group has trivial
homological determinant exactly when $\ell/t$ divides $2$
(Proposition~\ref{prop:hdet}), so with Theorem~A(v) the two conditions in
\cite[Question~6.2(2)]{ChanGaddisWonZhang2024OzoneGroupsCenters} (repeated as
\cite[Question~5.2(3)]{ChanGaddisWonZhang2025}) separate precisely when
$\ell/t\mid2$ and $\omega_1\omega_2\neq1$. For down-up algebras this is also
\cite[Question~6.7]{ChanGaddisWonZhang2024OzoneGroupsCenters}, that is,
\cite[Question~5.6]{ChanGaddisWonZhang2025}. Liu, Wu and Zhu had already
answered it with $A(0,1)$ \cite[Example~3.9]{LiuWuZhu2026OzoneGroups}. The
pair $p=\zeta_5$, $q=\zeta_5^{2}$ does it with a center that is not even
Gorenstein (Corollary~\ref{cor:hdet-noCY}). In Theorem~A(v), the case $p=q$ is due to Lopes, H.~Su\'arez and Y.~Su\'arez \cite[Corollary~3.4]{LopesSuarezSuarez2025};
the converse needs the Nakayama automorphism.

Theorem~B(i) is the presentation missing in
\cite[Section~5.1]{ChanGaddisWonZhang2025}. As a set of generators it is
Zhao's, but that set is not minimal. A minimal one needs the Hilbert basis of
the affine semigroup $L_{\geq0}$ of exponents of central monomials in $z$ and
$\theta$, not just generators of the lattice $L$ it spans. The center needs
exactly $|\mathcal B|+2-\delta$ generators (Corollary~\ref{cor:generadores}),
a number that is unbounded in the family, and it is a hypersurface exactly
when this number is at most four. This corrects
\cite[Theorem~4.1]{BeraMandalNandy2024}, whose list already falls short for
$p=\zeta_5$, $q=\zeta_5^{2}$. It also finishes the computation of
$Z(\mathcal H_q)$ begun in \cite[Theorem~3.5]{LV20} (Remark~\ref{rem:LV}).

Theorem~B(ii) disproves a conjecture of Chan, Gaddis, Won and Zhang. They
observed that $Z(A(0,1))$ is a polynomial ring
\cite[Example~4.8(1)]{ChanGaddisWonZhang2025} and conjectured that $A(0,1)$
is the only down-up algebra with regular center
\cite[p.~5724]{ChanGaddisWonZhang2024OzoneGroupsCenters}.
But every root of
unity $r\neq1$ gives another one, $A(1+r,-r)$, with roots $1$ and $r$:
$A(0,1)$ is $r=-1$, and the next is $r=\zeta_3$. These turn out to be the PI
generalized Heisenberg algebras of L\"u and Zhao \cite{LuZhao2015}
(Corollary~\ref{cor:Hf}). Theorem~B(iii) follows from Watanabe's theorem once the
pseudo-reflections are factored out of the action on
$\K[\Omega_1,\Omega_2]$. It produces noetherian PI down-up algebras whose
center is not Gorenstein, the smallest at $p=\zeta_5$, $q=\zeta_5^{2}$.
Together, parts~(ii) and~(iii) of Theorem~B answer
\cite[Question~6.4]{ChanGaddisWonZhang2024OzoneGroupsCenters} and the problem
left in \cite[Section~5.1]{ChanGaddisWonZhang2025}. When $n_1=n_2$ the
condition in Theorem~B(iii) is just $\omega_1\omega_2=1$, so for roots of the same
order a Gorenstein center and the Calabi--Yau property are the same thing.

The algebras with regular center have one more feature. When
$\alpha+\beta=1$ the algebra is generated in degree one and its center is a
polynomial ring on generators of degree greater than one. These are the
hypotheses of \cite[Question~6.2(5)]{ChanGaddisWonZhang2024OzoneGroupsCenters},
which asks whether all automorphisms are then affine and whether there are no
nonzero locally nilpotent derivations. Both answers are negative for every
such algebra (Proposition~\ref{prop:LND}, Corollary~\ref{cor:LND-downup}).

For $pq=1$ we have $\theta=z$, the algebra is not module-finite over its
center, and the ozone group is infinite. For $p=q=1$, the enveloping algebra
$U(\mathfrak h)$, it is not even abelian. Infinite nonabelian ozone groups
were known for a non-connected algebra
\cite[Example~5.1]{ChanGaddisWonZhang2025}; here the algebra is connected
graded and AS regular, but not PI.

This paper belongs to a series on ozone groups of iterated Ore extensions.
In \cite{GomezVenegas1} we treat the algebras $B_q(f)$ of Gaddis and Yee
\cite{GaddisYee2026TrivialOzone}, which, like $\mathcal H_{p,q}$, have global
dimension three, and in \cite{GomezVenegas2} the positive part
$U_q^{+}(B_2)$, which is Artin--Schelter regular of global dimension four.
Unlike $B_q(f)$ and $U_q^{+}(B_2)$, whose ozone groups are trivial for
suitable parameters, the algebras $\mathcal H_{p,q}$ never have trivial ozone
group.

Section~\ref{sec:down vs quant} sets up the dictionary between
$\mathcal H_{p,q}$ and down-up algebras. Section~\ref{sec:center} contains
Theorem~B. Section~\ref{sec:ozono} computes the ozone group, its fixed ring
and the normal elements, and Section~\ref{sec:pq1} treats $pq=1$.
Section~\ref{sec:CGWZ} collects the answers to the questions of Chan, Gaddis,
Won and Zhang. Throughout, $\K$ is an algebraically closed field of
characteristic zero.
\section{Graded down-up algebras versus two-parameter quantum Heisenberg algebras}\label{sec:down vs quant}

 In this section we present the relationship between graded down-up algebras and two-parameter quantum Heisenberg algebras. The monomials $z^{i}x^{j}y^{k}$ with $i,j,k\geq0$ form a PBW basis of
$\mathcal H_{p,q}$. We use two gradings for $\mathcal H_{p,q}$. The first is the $\Z^{2}$-grading given by
\begin{align}
\deg x=(1,0),\qquad \deg y=(0,1),\qquad \deg z=(1,1),
\label{eq:bigrad}
\end{align}
all three defining relations being homogeneous; the total grading with
$\deg x=\deg y=1$ and $\deg z=2$ is its specialization. Second, the filtration
with $|x|=|y|=|z|=1$ has associated graded the quantum affine space
\begin{align*}
\gr\mathcal H_{p,q}=\K\langle X,Y,Z\rangle/\bigl(ZX-p^{-1}XZ,\ ZY-pYZ,\ YX-qXY\bigr),
\end{align*}
because the only correction term, the $z$ in the third relation, reduces the
degree by one.

The element $z$ is normal, and so is
\begin{align}
\theta:=yx-p^{-1}xy=(q-p^{-1})xy+z,
\qquad
\theta z=z\theta,\quad \theta x=qx\theta,\quad \theta y=q^{-1}y\theta .
\label{eq:theta}
\end{align}
Assume $pq\neq1$ and write $\pqnum{k}=(q^{k}-p^{-k})/(q-p^{-1})$. Then
\cite[Lemma~2.2]{BeraMandalNandy2024}
\begin{align}
yx^{k}&=q^{k}x^{k}y+\pqnum{k}x^{k-1}z,
\label{eq:k1}\\
y^{k}x&=q^{k}xy^{k}+\pqnum{k}zy^{k-1},\nonumber
\end{align}
and
\begin{align}
\pqnum{k}=0\iff(pq)^{k}=1 .
\label{eq:pq}
\end{align}
Hence $x^{t}$ and $y^{t}$ are normal, where $t=\ord(pq)$.

\begin{lemma}\label{lem:anti}
The assignment $x\mapsto y$, $y\mapsto x$, $z\mapsto z$ extends to an
anti-automorphism $\tau$ of $\mathcal H_{p,q}$. It interchanges the bidegrees $(u,v)$ and
$(v,u)$, and $\tau(\theta)=\theta$.
\end{lemma}

\begin{proof}
Applying $\tau$ to the defining relations and reversing products turns
$zx-p^{-1}xz$ into $yz-p^{-1}zy$, $zy-pyz$ into $xz-pzx$, and $yx-qxy-z$ into
itself; all three hold in $\mathcal H_{p,q}$. Since $\tau^{2}$ is the identity on generators,
$\tau$ is bijective, and $\tau(\theta)=yx-p^{-1}xy=\theta$.
\end{proof}

We next record the diagonal automorphisms of $\mathcal H_{p,q}$ for arbitrary $p,q\in\K^\times$.

\begin{lemma}\label{lem:phi}
Let $p,q\in\K^{\times}$. For $a,b\in\K^{\times}$ the assignment
\begin{align}
\varphi_{a,b}(x)=ax,\qquad
\varphi_{a,b}(y)=by,\qquad
\varphi_{a,b}(z)=ab\,z
\label{eq:phi}
\end{align}
extends to a $\Z^{2}$-graded automorphism of $\mathcal H_{p,q}$, and
$(a,b)\mapsto\varphi_{a,b}$ is an injective group homomorphism
$(\K^{\times})^{2}\to\Aut(\mathcal H_{p,q})$. Moreover $\varphi_{a,b}(\theta)=ab\,\theta$
and
\begin{align}
\varphi_{a,b}\bigl|_{(\mathcal H_{p,q})_{(u,v)}}=a^{u}b^{v}
\qquad\text{for every }(u,v)\in\Z^{2}_{\geq0}.
\label{eq:peso}
\end{align}
\end{lemma}

\begin{proof}
The relations \eqref{eq:rel} are preserved:
\begin{align*}
\varphi_{a,b}(z)\varphi_{a,b}(x)&=a^{2}b\,zx=p^{-1}a^{2}b\,xz
=p^{-1}\varphi_{a,b}(x)\varphi_{a,b}(z),\\
\varphi_{a,b}(z)\varphi_{a,b}(y)&=ab^{2}\,zy=p\,ab^{2}\,yz
=p\,\varphi_{a,b}(y)\varphi_{a,b}(z),\\
\varphi_{a,b}(y)\varphi_{a,b}(x)-q\,\varphi_{a,b}(x)\varphi_{a,b}(y)
&=ab(yx-qxy)=ab\,z=\varphi_{a,b}(z).
\end{align*}
On generators $\varphi_{a,b}\varphi_{c,d}=\varphi_{ac,bd}$, so the assignment is
multiplicative and each $\varphi_{a,b}$ is bijective with inverse
$\varphi_{a^{-1},b^{-1}}$; it is injective because $\varphi_{a,b}=\mathrm{id}$
forces $a=b=1$. For \eqref{eq:peso}, a PBW monomial $x^{\alpha}z^{\beta}y^{\gamma}$
of bidegree $(u,v)$ has $\alpha=u-\beta$ and $\gamma=v-\beta$, so
\begin{align*}
\varphi_{a,b}(x^{\alpha}z^{\beta}y^{\gamma})
=a^{u-\beta}(ab)^{\beta}b^{v-\beta}\,x^{\alpha}z^{\beta}y^{\gamma}
=a^{u}b^{v}\,x^{\alpha}z^{\beta}y^{\gamma},
\end{align*}
and these monomials span $(\mathcal H_{p,q})_{(u,v)}$; outside $\Z^{2}_{\geq0}$
the components vanish. Finally $\theta$ has
bidegree $(1,1)$ by \eqref{eq:theta} and \eqref{eq:bigrad}.
\end{proof}

From now on $p$ and $q$ are roots of unity. Put
\begin{align}
L=\{(i,j)\in\Z^{2}:p^{i}=q^{j}\},\qquad
L_{\geq0}=L\cap\Z^{2}_{\geq0},\qquad
E=\{i+j:(i,j)\in L_{\geq0}\}.
\label{eq:L}
\end{align}

\begin{lemma}\label{lem:aritmetica}
Let $p,q$ be roots of unity with $pq\neq1$, let $\Gamma=\langle p,q\rangle$,
$\ell=|\Gamma|=\mathrm{lcm}(m,n)$ and $t=\ord(pq)$, and let $L$, $L_{\geq0}$
and $E$ be as in \eqref{eq:L}. Then $\gcd E=\ell/t$. Moreover, if $(r,s)$ is any
pair of nonnegative integers such that $p^{r}=q^{s}$ generates
$\langle p\rangle\cap\langle q\rangle$, then $\gcd(m,n,r+s)=\ell/t$; in
particular that number does not depend on the choice of $(r,s)$. Finally,
writing
\begin{align}
\Gamma=\langle g\rangle,\qquad p=g^{A},\qquad q=g^{B},
\label{eq:g}
\end{align}
and $e=\ell/t$, one has
\begin{align*}
\mathrm{lcm}(n,t)=\ell,\qquad e=\gcd(\ell,A+B),\qquad \gcd(e,A)=\gcd(e,B)=1 .
\end{align*}
\end{lemma}

\begin{proof}
Let $\pi:\Z^{2}\to\Z$, $\pi(i,j)=i+j$. Since $(m,0),(0,n)\in L_{\geq0}$, every
element of $L$ becomes nonnegative after adding a suitable combination of them,
so $L_{\geq0}$ generates $L$ as a group and $\pi(L)=\gcd E\cdot\Z$. Now
$c\in\pi(L)$ means $p^{i}=q^{c-i}$ for some $i$, that is $(pq)^{i}=q^{c}$; so
$c\in\pi(L)$ if and only if $q^{c}\in\langle pq\rangle$. In
$\overline\Gamma:=\Gamma/\langle pq\rangle$ we have
$\overline p=\overline q^{-1}$, hence
$\overline\Gamma=\langle\overline q\rangle$ and
$|\overline\Gamma|=\ord(\overline q)=\ell/t$. Therefore $\pi(L)=(\ell/t)\Z$.

For the second claim, $(m,0),(0,n),(r,s)\in L_{\geq0}$ gives
$\gcd E\mid\gcd(m,n,r+s)$. Conversely, if $(i,j)\in L$ then
$p^{i}\in\langle p\rangle\cap\langle q\rangle=\langle p^{r}\rangle$, so
$i\equiv rk\pmod m$ and $j\equiv sk\pmod n$ for a common $k\in\Z$; hence
$i+j\equiv(r+s)k$ modulo $\gcd(m,n)$, and $\gcd(m,n,r+s)$ divides $i+j$. For the last claim, $n=\ell/\gcd(B,\ell)$ and $t=\ell/\gcd(A+B,\ell)$, so
$e=\gcd(\ell,A+B)$ and
\begin{align*}
\mathrm{lcm}(n,t)=\frac{\ell}{\gcd\bigl(\gcd(B,\ell),\gcd(A+B,\ell)\bigr)}
=\frac{\ell}{\gcd(A,B,\ell)}=\ell .
\end{align*}
If a prime $\rho$ divides $e$ and $A$, then $\rho\mid\ell$ and
$\rho\mid A+B$, hence $\rho\mid B$ and $\rho\mid\gcd(A,B,\ell)=1$; so
$\gcd(e,A)=1$, and symmetrically $\gcd(e,B)=1$.
\end{proof}

\noindent From now on we write $e:=\ell/t$.

\begin{lemma}[{\cite[Lemma~1.9]{ChanGaddisWonZhang2025}}]\label{lem:19}
Let $A$ be prime and finitely generated as a module over its center.
\begin{enumerate}[\rm (i)]
\item Every $\sigma\in\Oz(A)$ equals $\eta_a$ for some regular normal
$a\in A$, where $\eta_a(w)=a^{-1}wa$.
\item If in addition $A$ is a domain graded by an ordered abelian group, then
$a$ may be taken homogeneous.
\item If $A$ is a $\Z^{n}$-graded domain, then $\Oz_{\mathrm{gr}}(A)=\Oz(A)$.
\end{enumerate}
\end{lemma}

\begin{proposition}[Posner, {\cite[Theorem~13.6.5]{mcconnell}}, in the form of
{\cite[Proposition~2.1]{LiuWuZhu2026OzoneGroups}}]\label{prop:posner}
Let $A$ be a prime PI ring with center $Z$ and let $K=\mathrm{Frac}(Z)$. Then
$A\otimes_ZK$ is a central simple $K$-algebra. In particular
$\rk_ZA=\dim_K(A\otimes_ZK)=\PIdeg(A)^{2}$.
\end{proposition}

For $\alpha,\beta,\gamma\in\K$, the down-up algebra of Benkart and Roby
\cite{BenkartRoby1998} is
\begin{align*}
A(\alpha,\beta,\gamma)=\K\langle d,u\rangle/\bigl(d^{2}u-\alpha dud-\beta ud^{2}-\gamma d,\
du^{2}-\alpha udu-\beta u^{2}d-\gamma u\bigr),
\end{align*}
and $d\mapsto y$, $u\mapsto x$ identifies $A(\alpha,\beta,0)$ with the graded
down-up algebra $A(\alpha,\beta)$ of the introduction.

Let $A=A(\alpha,\beta)$ be a graded down-up algebra. We put
$\Delta=\sqrt{\alpha^{2}+4\beta}$ and let $\omega_{1},\omega_{2}$ be the roots of
$X^{2}-\alpha X-\beta$, so that $\omega_1+\omega_2=\alpha$ and $\omega_1\omega_2=-\beta$. When
$\beta\neq0$ the algebra is a noetherian domain
\cite[Main Theorem]{KirkmanMussonPassman1999Noetherian}, and
$\{x^{i}(yx)^{j}y^{k}\}$ is a PBW basis. Since $\mathcal H_{p,q}\cong A\bigl(p^{-1}+q,-p^{-1}q\bigr)$ and $X^{2}-(p^{-1}+q)X+p^{-1}q=(X-q)(X-p^{-1})$, the characteristic roots are $\omega_{1}=q$ and $\omega_{2}=p^{-1}$, and
the condition $pq\neq1$ corresponds to the condition $\Delta\neq0$. Conversely, given $A(\alpha,\beta)$ with $\beta\neq0$ and $\Delta\neq0$, we put
\begin{align*}
q:=\omega_{1},\qquad p:=\omega_{2}^{-1},
\end{align*}
so that $pq\neq1$ and $A(\alpha,\beta)\cong\mathcal H_{p,q}$. When $\omega_1,\omega_2$
are roots of unity, Lemma~\ref{lem:aritmetica} applies, with
$n=\ord(\omega_{1})$, $m=\ord(\omega_{2})$, $\ell=\mathrm{lcm}(m,n)$ and
$t=\ord(\omega_{1}\omega_{2}^{-1})$. We will use this correspondence throughout.

In the notation of \cite[Section~7]{KirkmanKuzmanovichZhang2015DownUp}, where
the roots are called $a,b$ and $\Omega_1=du-a\,ud$, $\Omega_2=du-b\,ud$, the
choice $a=q$, $b=p^{-1}$ gives
\begin{align*}
\Omega_1&=yx-qxy=z, & \Omega_2&=yx-p^{-1}xy=\theta,
\end{align*}
and the commutation rules of \cite[Lemma~7.1]{KirkmanKuzmanovichZhang2015DownUp}
are exactly the relations $zx=p^{-1}xz$, $zy=pyz$ of \eqref{eq:rel} and those of
\eqref{eq:theta}. We use this notation whenever we compare with their results.

\begin{remark}\label{rem:simetria}
The pair $(\alpha,\beta)$ depends only on the set of roots, so exchanging
$\omega_1$ and $\omega_2$ gives $\mathcal H_{p,q}\cong\mathcal H_{q^{-1},p^{-1}}$. Every
invariant computed below --- $\ell$, $t$, the Gorenstein condition, the
regularity condition and the Calabi--Yau condition --- is unchanged under
$(p,q)\mapsto(q^{-1},p^{-1})$.
Gaddis also showed that $\mathcal H_{p,q}\cong\mathcal H_{p',q'}$ forces
$(p',q')\in\{(p,q),(q,p),(p^{-1},q^{-1}),(q^{-1},p^{-1})\}$
\cite[Proposition~3.14]{Gaddis2016}. The other two pairs invert both
characteristic roots, which corresponds to passing to the opposite algebra, and
$\mathcal H_{p,q}\cong\mathcal H_{p,q}^{\mathrm{op}}$ through the anti-automorphism of
Lemma~\ref{lem:anti}. The invariants listed above are also unchanged under
$(p,q)\mapsto(q,p)$, as they should be.
\end{remark}

\begin{proposition}\label{prop:Hpq-downup}
Up to isomorphism, the algebras $\mathcal H_{p,q}$ with $p,q$ roots of unity
and $pq\neq1$ are precisely the noetherian PI graded down-up algebras
$A(\alpha,\beta)$.
\end{proposition}

\begin{proof}
By \cite[Main Theorem]{KirkmanMussonPassman1999Noetherian},
$A(\alpha,\beta,\gamma)$ is noetherian if and only if $\beta\neq0$, and by
\cite[Theorem~2.4]{CarvalhoLompPusatYilmaz2010Injective}, a noetherian graded down-up algebra $A(\alpha,\beta)$ is PI if and only if
the roots of $X^{2}-\alpha X-\beta$ are distinct roots of unity.
 If $A(\alpha,\beta)$ is noetherian and PI, its characteristic
roots $\omega_1,\omega_2$ are therefore distinct roots of unity, and taking
$q=\omega_1$, $p=\omega_2^{-1}$ gives $pq\neq1$ and
$A(\alpha,\beta)\cong\mathcal H_{p,q}$. Conversely, if $p,q$ are roots of unity
and $pq\neq1$, then $\mathcal H_{p,q}\cong A(p^{-1}+q,-p^{-1}q)$ has
$\beta=-p^{-1}q\neq0$ and distinct characteristic roots $q$ and $p^{-1}$.
\end{proof}

\begin{remark}\label{rem:gamma0}
The hypothesis $\gamma=0$ matters. For $\gamma\neq0$ one may normalize
$\gamma=1$, and those algebras are filtered rather than graded; the PI criterion changes too: when one characteristic root equals $1$, the
algebra $A(\alpha,\beta,1)$ is no longer PI, even though both roots are distinct
roots of unity --- no power of $u$ is central, and the whole center sits in
degree zero of Zhao's $\Z$-grading, which has GK dimension two. We come back to
the case $\gamma\neq0$ elsewhere.  Everything below uses the grading and is
therefore restricted to $\gamma=0$. Note also that $p=1$ or $q=1$ is allowed,
as long as $pq\neq1$.
\end{remark}

Finally, the assignment $d\mapsto u$, $u\mapsto d$ extends to an
anti-automorphism $\tau$ of $A(\alpha,\beta,\gamma)$ for every $\gamma$: it
turns each defining relation into the other one. For $\gamma=0$ it is the map
of Lemma~\ref{lem:anti}.

\section{The center of \texorpdfstring{$\mathcal H_{p,q}$}{Hpq}}\label{sec:center}

Throughout this section $p,q\in\K^{\times}$ are roots of unity with $pq\neq1$,
and $m,n,\ell,t,e$ and $L,L_{\geq0},E$ are as in \eqref{eq:L} and
Lemma~\ref{lem:aritmetica}. Bera, Mandal and Nandy
\cite[Theorem~3.1]{BeraMandalNandy2024} showed that $\PIdeg(\mathcal H_{p,q})=\ell$.

\subsection{The center as a bigraded space}

Recall the $\Z^{2}$-grading \eqref{eq:bigrad}, for which $\deg\theta=(1,1)$ as
well. Since $z$ is normal, the monomials $x^{a}z^{b}y^{c}$ form a PBW basis just
as $z^{a}x^{b}y^{c}$ do, and we use the former.

\begin{lemma}\label{lem:factor}
Assume $pq\neq1$. For every $k\geq1$,
\begin{align*}
\prod_{j=0}^{k-1}\bigl(\theta-(pq)^{j}z\bigr)
&=q^{\binom{k}{2}}(q-p^{-1})^{k}\,x^{k}y^{k}.
\end{align*}
In particular,
\begin{align*}
\theta^{k}
&\equiv q^{\binom{k}{2}}(q-p^{-1})^{k}\,x^{k}y^{k}
\pmod{z\mathcal H_{p,q}},\\
\bigl(\theta^{t}-z^{t}\bigr)^{\ell/t}
&=q^{\binom{\ell}{2}}(q-p^{-1})^{\ell}\,x^{\ell}y^{\ell}.
\end{align*}
\end{lemma}

\begin{proof}
Put $c=q-p^{-1}$, which is nonzero because $pq\neq1$. Since $\theta-z=c\,xy$ by
\eqref{eq:theta}, $\theta y=q^{-1}y\theta$ and $zy=p\,yz$,
\begin{align*}
x^{k}y^{k}
&=c^{-1}x^{k-1}(\theta-z)\,y^{k-1}=c^{-1}x^{k-1}y^{k-1}\bigl(q^{-(k-1)}\theta-p^{k-1}z\bigr).
\end{align*}
Since $\theta$ and $z$ commute, induction on $k$ gives
\begin{align*}
x^{k}y^{k}
&=c^{-k}\prod_{j=0}^{k-1}\bigl(q^{-j}\theta-p^{j}z\bigr)\\
&=c^{-k}\prod_{j=0}^{k-1}q^{-j}\bigl(\theta-(pq)^{j}z\bigr)\\
&=q^{-\binom{k}{2}}c^{-k}\prod_{j=0}^{k-1}\bigl(\theta-(pq)^{j}z\bigr).
\end{align*}
The congruence follows by reducing modulo $z\mathcal H_{p,q}$. For the last
identity take $k=\ell$. As $j$ runs from $0$ to $\ell-1$, the scalar $(pq)^{j}$
runs over $\mu_t$ exactly $\ell/t$ times, and since $\theta$ and $z$ commute,
\begin{align*}
\prod_{\xi\in\mu_t}(\theta-\xi z)&=\theta^{t}-z^{t}.
\end{align*}
\end{proof}

\noindent The first identity of Lemma~\ref{lem:factor} is
\cite[Lemma~8.1]{KirkmanKuzmanovichZhang2015DownUp}, the term $u_k d_k = f_k(\Omega_1, \Omega_2)$ appearing there becomes our product formula. We kept the short proof because the case $k=\ell$, where the factors group into
$\theta^{t}-z^{t}$, is the one used below.

\begin{lemma}\label{lem:comp}
For $pq\neq 1$ and $u\geq v\geq0$ the elements $x^{u-v}z^{i}\theta^{j}$ with $i+j=v$ form a
basis of $(\mathcal H_{p,q})_{(u,v)}$; symmetrically for $v\geq u$ with
$z^{i}\theta^{j}y^{v-u}$ and $i+j=u$. In particular $z$ and $\theta$ commute
and are algebraically independent, so $\K[z,\theta]$ is a polynomial subalgebra
of $\mathcal H_{p,q}$.
\end{lemma}

\begin{proof}
The PBW monomials of bidegree $(u,v)$ are $x^{u-b}z^{b}y^{v-b}$ with
$0\leq b\leq\min(u,v)$, so the dimension is $\min(u,v)+1$, which is the number
of proposed elements. By Lemma~\ref{lem:factor} we may write
$\theta^{j}=c_{j}x^{j}y^{j}+z(\cdots)$ with $c_{j}\neq0$, whence
\begin{align*}
x^{u-v}z^{i}\theta^{j}
=c_{j}p^{-ij}\,x^{u-v+j}z^{i}y^{j}+(\text{terms of $z$-degree}>i).
\end{align*}
For $i+j=v$ the leading PBW monomials are pairwise distinct, so the transition
matrix is triangular for the $z$-degree filtration. The case $v\geq u$ follows by applying $\tau$ of Lemma~\ref{lem:anti}.
\end{proof}

\begin{proposition}\label{prop:centro-bigrad}
$Z(\mathcal H_{p,q})$ has $\K$-basis
\begin{align*}
\{\,x^{\ell a}z^{i}\theta^{j}:a\geq0,\ (i,j)\in L_{\geq0}\,\}
\ \cup\
\{\,y^{\ell b}z^{i}\theta^{j}:b\geq1,\ (i,j)\in L_{\geq0}\,\}.
\end{align*}
\end{proposition}

\begin{proof}
The center of a $\Z^{2}$-graded algebra is a graded subspace, so by
Lemma~\ref{lem:comp} it suffices to decide when $f=x^{a}g$ is central, with
$a=u-v\geq0$ and $g=\sum c_{ij}z^{i}\theta^{j}$ homogeneous of degree $v$ in
$\K[z,\theta]$; the case $v\geq u$ is symmetric by Lemma~\ref{lem:anti}.

From $z^{i}\theta^{j}x=p^{-i}q^{j}\,x\,z^{i}\theta^{j}$ we get
$fx=x^{a+1}\sum c_{ij}p^{-i}q^{j}z^{i}\theta^{j}$ while $xf=x^{a+1}g$, so
$fx=xf$ forces $p^{-i}q^{j}=1$, that is $(i,j)\in L$, throughout the support of
$g$.

On that support $z^{i}\theta^{j}y=p^{i}q^{-j}y\,z^{i}\theta^{j}
=y\,z^{i}\theta^{j}$, so $g$ commutes with $y$ and $fy=x^{a}yg$. If $a=0$ then $f=g$ is already central and $\ell\mid a$ holds trivially,
so assume $a\geq1$. On the other
hand \eqref{eq:k1} gives $yx^{a}=q^{a}x^{a}y+\pqnum{a}x^{a-1}z$, and
$x^{a}y=(q-p^{-1})^{-1}x^{a-1}(\theta-z)$ for $a\geq1$, so
\begin{align*}
fy-yf=x^{a-1}\Bigl[\tfrac{1-q^{a}}{q-p^{-1}}\,(\theta-z)-\pqnum{a}\,z\Bigr]g .
\end{align*}
Since $\mathcal H_{p,q}$ is a domain and $g\neq0$, this vanishes if and only if the bracket
does. The bracket is a linear form in the polynomial ring $\K[z,\theta]$: its
$\theta$-coefficient is $(1-q^{a})(q-p^{-1})^{-1}$, whose vanishing gives
$q^{a}=1$, and then its $z$-coefficient reduces to $-\pqnum{a}$, whose
vanishing gives $(pq)^{a}=1$ by \eqref{eq:pq}. So $n\mid a$ and $t\mid a$.
Hence $\ell\mid a$, by Lemma~\ref{lem:aritmetica}.

Conversely, an element $x^{\ell a}z^{i}\theta^{j}$ with $(i,j)\in L_{\geq0}$
commutes with $x$ and with $y$ by the computations above, and with $z$ because
$zx^{\ell a}=p^{-\ell a}x^{\ell a}z=x^{\ell a}z$; so it is central.
\end{proof}

\begin{corollary}\label{cor:soporte}
The support of $Z(\mathcal H_{p,q})$ in $\Z^{2}$ is
\begin{align*}
\{(\ell a+\varepsilon,\,\varepsilon):a\geq0,\ \varepsilon\in E\}
\ \cup\
\{(\varepsilon,\,\ell b+\varepsilon):b\geq1,\ \varepsilon\in E\},
\end{align*}
and $\gcd E=\ell/t=e$.
\end{corollary}

\begin{proof}
The assertions follow directly from Proposition~\ref{prop:centro-bigrad} and Lemma~\ref{lem:aritmetica}, respectively.
\end{proof}

\begin{remark}[Comparison with Zhao's description of the center]
\label{rem:zhao}
The description of $Z(\mathcal H_{p,q})$ in
Proposition~\ref{prop:centro-bigrad} is consistent with
\cite[Theorem~1.3(f)]{Zhao1999}. Indeed, under the isomorphism
$\mathcal H_{p,q}\cong A(p^{-1}+q,-p^{-1}q)$, the characteristic
roots are $\omega_1=q$ and $\omega_2=p^{-1}$, and the corresponding normal
elements are
\begin{align*}
du-\omega_1ud=z,\qquad du-\omega_2ud=\theta.
\end{align*}
Consequently, the condition on the exponents of the central
monomials in Zhao's description becomes
$p^i=q^j$, which is precisely the defining condition of
the lattice $L$ in \eqref{eq:L}.

Our description additionally determines the bigraded components
of the center. Corollary~\ref{cor:generadores} gives a finite
generating set in terms of the Hilbert basis of $L_{\geq0}$,
while Proposition~\ref{prop:presentacion} provides an explicit
presentation of the center.
\end{remark}

\subsection{On the generators of the center}

The center of $\mathcal H_{p,q}$ is described in \cite[Section~4]{BeraMandalNandy2024}
as well: Theorem~4.1 there lists $z^{m}$, $\theta^{n}$, $x^{\ell}$, $y^{\ell}$
and $\Omega=z^{r}\theta^{s}$ as generators, where $(r,s)$ is a pair with
$p^{r}=q^{s}$ generating $\langle p\rangle\cap\langle q\rangle$. That list is
not always complete, the pair $(r,s)$
generates $L$ modulo $m\Z\times n\Z$ \emph{as a group}, whereas the central
monomials $z^{i}\theta^{j}$ are the lattice points of the positive quadrant, and
as a \emph{semigroup} those may need more generators. The next remark records
what has to be added, and Corollary~\ref{cor:generadores} gives the exact number
of generators required.

\begin{remark}\label{rem:erratum}
The outcome depends on the choice of $(r,s)$. Take $p=q=\zeta_3$, so $m=n=\ell=t=3$. With
$(r,s)=(1,1)$ one does get all of $Z(\mathcal H_{p,q})$. The pair $(r,s)=(2,2)$
is equally admissible, since $p^{2}=q^{2}$ generates
$\langle p\rangle\cap\langle q\rangle=\mu_3$, but then $z\theta$, which is
central of bidegree $(2,2)$, is missing: the subalgebra
$\K[z^{3},\theta^{3},x^{3},y^{3},z^{2}\theta^{2}]$ is zero in that bidegree, its
generators sitting in bidegrees $(3,3)$, $(3,0)$, $(0,3)$ and $(4,4)$. Since no
minimality on $(r,s)$ is imposed, the choice matters.

For some parameters no admissible choice works. Take $p=\zeta_5$ and $q=\zeta_5^{2}$, so
$m=n=\ell=t=5$ and $e=1$. Then
\begin{align*}
L_{\geq0}=\langle(2,1),\,(1,3),\,(5,0),\,(0,5)\rangle,
\end{align*}
a Hilbert basis with four elements, and the admissible pairs are exactly
$(2,1),(4,2),(1,3),(3,4)$. Each choice leaves out one of the two Hilbert basis
elements off the axes: for $(r,s)=(2,1)$ the element $z\theta^{3}$ is central
and is not in the subalgebra generated by the five, and for the other three
choices the same happens with $z^{2}\theta$. Indeed, in the grading of
$\K[z,\theta]$ the five elements contribute only in degrees $r+s$, $5$ and $10$,
whereas $z^{2}\theta$ and $z\theta^{3}$ sit in degrees $3$ and $4$, and
$Z(\mathcal H_{p,q})$ is one-dimensional in each of those two degrees. The same
happens for $q=\zeta_5^{3}$.

The number of generators changes as follows. Take $(r,s)=(2,1)$ in the example
above. The \emph{number} of generators is still five: by Corollary~\ref{cor:generadores}
the minimum is $|\mathcal B|+2-\delta=4+2-1=5$. What changes is the list, since
$z\theta^{3}$ has to be added and $\theta^{5}$ then becomes redundant, being
$\theta^{5}=z^{5}+c^{-1}x^{5}y^{5}$ by Lemma~\ref{lem:factor}. In general the
number is not five: it is four for $p=q$ (Corollary~\ref{cor:centro-Hq}), and it
is unbounded. For $\ell=16$, $p=\zeta_{16}^{-1}$ and $q=\zeta_{16}^{2}$, the
Hilbert basis of $L_{\geq0}$ is
\begin{align*}
\{(0,8),(2,7),(4,6),(6,5),(8,4),(10,3),(12,2),(14,1),(16,0)\},
\end{align*}
so $Z(\mathcal H_{p,q})$ needs exactly ten generators.

The discrepancy originates in \cite[(4.3)]{BeraMandalNandy2024}, where the
nonnegative solutions of $p^{a}=q^{b}$ are described as
$(a,b)=(rk+m\Z_{\geq0},\,sk+n\Z_{\geq0})$ with $k\in\Z_{\geq0}$. The congruences
$a\equiv rk\ (\mathrm{mod}\ m)$ and $b\equiv sk\ (\mathrm{mod}\ n)$ are correct,
but the representative $rk$ need not be $\leq a$: for $(a,b)=(1,3)$ one is forced
to take $k=3$, and then $rk=6>1$. The reduction step in the proof subtracts a
multiple of $\Omega^{k}(\theta^{n})^{k'}x^{u-v}$, and that is exactly where the
nonnegativity is needed.

Two things are unaffected. The PI degree computation
\cite[Theorem~3.1]{BeraMandalNandy2024} does not use this description, and
neither does any statement below that only uses $\PIdeg(\mathcal H_{p,q})=\ell$.
Let us also note that the generating set of \cite[Theorem~1.3(f)]{Zhao1999},
which predates \cite{BeraMandalNandy2024} and is not quoted there, agrees with
Proposition~\ref{prop:centro-bigrad}; see Remark~\ref{rem:zhao}.
\end{remark}

\begin{remark}\label{rem:rs}
The derivation of the pair $(r,s)$ carried out in
\cite[Section~4]{BeraMandalNandy2024} is not needed here. By
Lemma~\ref{lem:aritmetica}, $\gcd(m,n,r+s)=\ell/t=e$ for every admissible pair,
so that invariant does not depend on the choice, whereas the generating set
discussed above does. We keep $\Omega=z^{r}\theta^{s}$ only when comparing with
\cite{BeraMandalNandy2024}; the role of
\cite[Theorem~4.1]{BeraMandalNandy2024} is played throughout by
Proposition~\ref{prop:centro-bigrad}, and in Section~\ref{sec:normales} by
nothing at all.
\end{remark}

\subsection{A presentation of the center}
We now give an explicit presentation of $Z(\mathcal H_{p,q})$. Let
\[
R:=\K[L_{\geq0}]
=\bigoplus_{(i,j)\in L_{\geq0}}\K z^i\theta^j
\subseteq\K[z,\theta],
\]
and define
\begin{align}
F:=q^{-\binom{\ell}{2}}(q-p^{-1})^{-\ell}
\bigl(\theta^t-z^t\bigr)^{\ell/t}.
\label{eq:F}
\end{align}

\begin{proposition}\label{prop:presentacion}
For $F$ as in \eqref{eq:F} one has $F\in R$, and
\begin{align*}
Z(\mathcal H_{p,q})\;\cong\;R[U,V]/(UV-F),\qquad U\mapsto x^{\ell},\quad V\mapsto y^{\ell}.
\end{align*}
Moreover $R$ is a two-dimensional normal Cohen--Macaulay domain: with $g$, $A$
and $B$ as in \eqref{eq:g}, it is the invariant ring $\K[z,\theta]^{\mu_\ell}$
for the action \eqref{eq:accion} below. In particular $Z(\mathcal H_{p,q})$ is a
hypersurface over a cyclic quotient surface singularity, and it is a
Cohen--Macaulay domain of Krull dimension three.
\end{proposition}

\begin{proof}
The monomials of $(\theta^{t}-z^{t})^{\ell/t}$ are $z^{ti}\theta^{\ell-ti}$, and
$p^{ti}=q^{\ell-ti}$ because $q^{\ell}=1$ and $(pq)^{ti}=1$; hence $F\in R$, and
Lemma~\ref{lem:factor} says exactly that $x^{\ell}y^{\ell}=F$. Since $x^{\ell}$,
$y^{\ell}$ and $R$ are central and commute with one another, sending
$U\mapsto x^{\ell}$, $V\mapsto y^{\ell}$ and fixing $R$ defines an algebra map
$R[U,V]\to Z(\mathcal H_{p,q})$ which kills $UV-F$. It is surjective by
Proposition~\ref{prop:centro-bigrad}, every basis element there being
$U^{k}z^{i}\theta^{j}$ or $V^{k}z^{i}\theta^{j}$ with $(i,j)\in L_{\geq0}$. For
injectivity, $R[U,V]/(UV-F)$ is free over $R$ on
$\{U^{k}\}_{k\geq0}\cup\{V^{k}\}_{k\geq1}$, and those elements times the
monomials of $R$ map bijectively onto the basis of
Proposition~\ref{prop:centro-bigrad}.

For the last statement, $L_{\geq0}=L\cap\Z^{2}_{\geq0}$ is saturated in $L$, so
$R$ is a normal semigroup ring, hence Cohen--Macaulay. The map
$\Z^{2}\to\Gamma$, $(i,j)\mapsto p^{i}q^{-j}=g^{Ai-Bj}$, is surjective with
kernel $L$, so it grades $\K[z,\theta]$ by $\Gamma\cong\mu_\ell$ with $R$ as its
degree-zero part; the corresponding $\mu_\ell$-action is \eqref{eq:accion}, and
it is faithful because $\gcd(A,B,\ell)=1$. Finally $UV-F$ is a nonzero divisor on
the Cohen--Macaulay ring $R[U,V]$ of dimension four, so the quotient is
Cohen--Macaulay of dimension three, and it is a domain because
$Z(\mathcal H_{p,q})$ sits inside the domain $\mathcal H_{p,q}$.
\end{proof}

\noindent Explicitly, the action just described is
\begin{align}
\zeta\cdot z=\zeta^{A}z,\qquad \zeta\cdot\theta=\zeta^{-B}\theta
\qquad(\zeta\in\mu_\ell),
\label{eq:accion}
\end{align}
and it is the one used again in Subsection~\ref{sec:gorenstein}.

\begin{corollary}\label{cor:generadores}
Let $\mathcal B$ be the Hilbert basis of $L_{\geq0}$, and let $\delta=1$ if
$(tk,\ell-tk)\in\mathcal B$ for some $0\leq k\leq e$, and $\delta=0$ otherwise. Then
\begin{align*}
Z(\mathcal H_{p,q})=\K\bigl[x^{\ell},\ y^{\ell},\ z^{i}\theta^{j}\ \ ((i,j)\in\mathcal B)\bigr],
\end{align*}
and $Z(\mathcal H_{p,q})$ needs exactly $|\mathcal B|+2-\delta$ homogeneous generators.
In particular $\delta=1$ whenever $m=\ell$ or $n=\ell$.
\end{corollary}

\begin{proof}
Generation follows from Proposition~\ref{prop:presentacion}. Let $\mathfrak m$ and $R_{+}$
be the irrelevant ideals of $Z(\mathcal H_{p,q})$ and $R$. Off the diagonal, $\mathfrak m/\mathfrak m^{2}$
is spanned by the classes of $x^{\ell}$ and $y^{\ell}$. On the diagonal, the only products
involving $x^{\ell}$ or $y^{\ell}$ are the elements $x^{\ell a}y^{\ell a}r=F^{a}r$ with $r\in R$, so
\begin{align*}
(\mathfrak m^{2})_{(k,k)}&=(R_{+}^{2})_{k}\qquad (k\neq\ell),\\
(\mathfrak m^{2})_{(\ell,\ell)}&=(R_{+}^{2})_{\ell}+\K F,\\
F&=c\sum_{k=0}^{e}(-1)^{k}\binom{e}{k}\,z^{tk}\theta^{\ell-tk},\qquad c\neq0 .
\end{align*}
The monomials $z^{i}\theta^{j}$ with $(i,j)\in\mathcal B$ form a basis of $R_{+}/R_{+}^{2}$, so
$F\notin R_{+}^{2}$ exactly when $\delta=1$. If $n=\ell$, no $(0,j)$ with $0<j<\ell$ lies in $L$,
hence $(0,\ell)\in\mathcal B$; the case $m=\ell$ is the same.
\end{proof}

For the one-parameter algebra the answer is completely explicit, and it is
shorter than the list of \cite[Theorem~4.1]{BeraMandalNandy2024}.

\begin{corollary}\label{cor:centro-Hq}
Let $q$ be a primitive $m$-th root of unity, $m\geq3$. Exchanging $x$ and $y$
identifies $\mathcal H_q$ with $\mathcal H_{q,q}$, and in the generators of
\eqref{eq:rel}
\begin{align*}
Z(\mathcal H_q)=\K\bigl[z\theta,\ z^{m},\ x^{m},\ y^{m}\bigr],
\end{align*}
and four generators are needed. Writing
$c=q^{-\binom{m}{2}}(q-q^{-1})^{-m}$,
\begin{align*}
\theta^{m}&=z^{m}+c^{-1}x^{m}y^{m} &&(m\ \text{odd}),\\
\theta^{m}&=2(z\theta)^{m/2}-z^{m}+c^{-1}x^{m}y^{m} &&(m\ \text{even}),
\end{align*}
and $Z(\mathcal H_q)$ is the hypersurface in $\mathbb A^{4}$ cut out by
$z^{m}\theta^{m}=(z\theta)^{m}$ after substituting these expressions.
\end{corollary}

\begin{proof}
Here $p=q$, so $L=\{(i,j):i\equiv j\ (\mathrm{mod}\ m)\}$, whose Hilbert basis
is $\mathcal B=\{(1,1),(m,0),(0,m)\}$ and $R=\K[z\theta,z^{m},\theta^{m}]$. Also
$t=m/\gcd(m,2)$ and $e=\gcd(m,2)$, so $F=c(\theta^{t}-z^{t})^{e}$ expands as
displayed, using $z^{m/2}\theta^{m/2}=(z\theta)^{m/2}$ when $m$ is even. In both
cases $\theta^{m}$ lies in $\K[z\theta,z^{m},x^{m}y^{m}]$, which gives
generation by the four listed elements. Since $(0,m)\in\mathcal B$ we have
$\delta=1$, and Corollary~\ref{cor:generadores} gives
$|\mathcal B|+2-\delta=4$. The relation $z^{m}\theta^{m}=(z\theta)^{m}$ holds in
$\K[z,\theta]$, and eliminating $\theta^{m}$ turns it into a single equation in
$z\theta,z^{m},x^{m},y^{m}$; by Proposition~\ref{prop:presentacion} there are no
others. Note $q\neq q^{-1}$, since $m\geq3$; for $m=2$ one would have
$pq=q^{2}=1$, the case of Section~\ref{sec:pq1}.
\end{proof}
\begin{remark}\label{rem:LV}
This closes a computation left open in \cite[Theorem~3.5]{LV20}. There
$\mathcal H_q$ is viewed as a skew PBW extension (for the skew PBW structure of
$\mathcal H_{p,q}$ over $\K[z]$ see also \cite[Section~2]{LopesSuarezSuarez2025}), with generators $x,y,z$ and
relations $yx=qxy$, $zy=qyz$, $zx=q^{-1}xz+y$, and it is shown that
$C=(q^{2}-1)xyz-y^{2}$ and the $m$-th powers of $x$, $y$, $z$ are central. Whether
they generate the center was not decided. Sending $x\mapsto x$, $y\mapsto z$,
$z\mapsto y$ identifies that algebra with $\mathcal H_{q^{-1},q^{-1}}$ in the
notation of \eqref{eq:rel},
and $C$ goes to $-z\theta$. So those four elements are, up to sign, the
generators of Corollary~\ref{cor:centro-Hq}, and the inclusion of
\cite{LV20} is an equality. The case $m=2$ has to be left out there: for $q=-1$
we have $pq=1$, neither $x^{2}$ nor $z^{2}$ is central in the notation of
\cite{LV20}, and the center is $\K[z^{2}]$ in ours by Proposition~\ref{prop:pq1}.
\end{remark}

\subsection{Regularity of the center}

Since $\mathcal H_{p,q}$ is module-finite over its center and $\GKdim\mathcal H_{p,q}=3$ (it is a bijective skew PBW extension of $\K[z]$ in $x,y$, see \cite[Section~2]{LopesSuarezSuarez2025} and, for skew PBW extensions in general, \cite{FGLRSV20}, with $\sigma_x(z)=pz$, $\sigma_y(z)=p^{-1}z$ and no derivations, so the frame $\K+\K z$ is preserved and \cite[Theorem~2.1(xiv)]{LV20c} gives $\GKdim\mathcal H_{p,q}=\GKdim\K[z]+2$),  the center $Z:=Z(\mathcal H_{p,q})$ is an affine domain of Krull dimension three. Moreover, $Z$ is connected $\mathbb N$-graded, with homogeneous maximal ideal $Z_+=\bigoplus_{d>0}Z_d$. Consequently, $Z$ is regular if and only if
\[
\dim_{\K}(Z_+/Z_+^2)=3,
\]
or equivalently, if and only if $Z$ admits three homogeneous algebra generators.

\begin{proposition}\label{prop:regular}
$Z(\mathcal H_{p,q})$ is regular if and only if $p=1$ or $q=1$, and then
\begin{align*}
Z(\mathcal H_{1,q})=\K[z,\,x^{n},\,y^{n}],\qquad
Z(\mathcal H_{p,1})=\K[\theta,\,x^{m},\,y^{m}],
\end{align*}
polynomial rings in three variables.
\end{proposition}

\begin{proof}
Both $(m,0)$ and $(0,n)$ lie in $\mathcal B$: a decomposition of $(m,0)$ inside
$L_{\geq0}$ uses only pairs $(i,0)$, and those have $m\mid i$. So
$|\mathcal B|=2$ forces $L_{\geq0}=\langle(m,0),(0,n)\rangle$, hence
$L=m\Z\times n\Z$, which says $\langle p\rangle\cap\langle q\rangle=1$, that is
$\gcd(m,n)=1$; the converse is clear.

If $\gcd(m,n)>1$, then $|\mathcal B|\geq3$ and Corollary~\ref{cor:generadores}
gives at least $3+2-1=4$ generators.

If $\gcd(m,n)=1$, then $\ord(pq)=mn$, so $t=\ell=mn$ and $e=1$, and the only
pairs $(tk,\ell-tk)$ available in Corollary~\ref{cor:generadores} are $(0,\ell)$
and $(\ell,0)$. Now $(0,mn)$ is indecomposable exactly when $m=1$, and $(mn,0)$
exactly when $n=1$; so $\delta=1$ precisely when $p=1$ or $q=1$, the number of
generators being $4-\delta$.

For the presentations take $p=1$, where $L=\Z\times n\Z$, so $R=\K[z,\theta^{n}]$
and $F=c(\theta^{n}-z^{n})$ with $c\neq0$ by \eqref{eq:F} and
Lemma~\ref{lem:factor}. Eliminating $\theta^{n}=z^{n}+c^{-1}UV$ from
Proposition~\ref{prop:presentacion} leaves $\K[z,U,V]$. The case $q=1$ is the
same with $z$ and $\theta$ exchanged.
\end{proof}

\noindent For $n=2$ we have $\mathcal H_{1,-1}\cong A(0,1)$, and this recovers
\cite[Example~4.8(1)]{ChanGaddisWonZhang2025}: $Z=\K[z,x^{2},y^{2}]$ with
$z=yx+xy$.

\begin{corollary}\label{cor:regular-downup}
Let $A=A(\alpha,\beta)$ be a noetherian PI graded down-up algebra. Then $Z(A)$ is
regular if and only if $\alpha+\beta=1$, that is, if and only if $1$ is a
characteristic root of $A$.
\end{corollary}

\begin{proof}
By Proposition~\ref{prop:Hpq-downup} we may take $A=\mathcal H_{p,q}$ with
$\alpha=p^{-1}+q$ and $\beta=-p^{-1}q$. Then
\begin{align*}
\alpha+\beta-1=p^{-1}+q-p^{-1}q-1=-(p^{-1}-1)(q-1),
\end{align*}
which vanishes exactly when $p=1$ or $q=1$, and $X^{2}-\alpha X-\beta$ has $1$ as
a root exactly when $\alpha+\beta=1$.
\end{proof}

The algebras with regular center have another name. For $f\in\K[w]$, the
\emph{generalized Heisenberg algebra} $\mathcal H(f)$ of L\"u and Zhao
\cite{LuZhao2015}, in the form used in \cite[Definition~2.6]{LopesSuarezSuarez2025},
is generated by $x,y,w$ with
\begin{align}
wx=xf,\qquad yw=fy,\qquad yx-xy=f-w .
\label{eq:Hf}
\end{align}
It is noetherian if and only if $\deg f=1$ \cite[Theorem~2.7]{LopesSuarezSuarez2025},
and $\mathcal H_{p,q}$ is isomorphic to some $\mathcal H(f)$ if and only if exactly
one of $p,q$ equals $1$ \cite[Theorem~2.9]{LopesSuarezSuarez2025}. Comparing with
Proposition~\ref{prop:regular}, the noetherian PI graded down-up algebras with regular
center are exactly the PI generalized Heisenberg algebras. Our results then
give the following.

\begin{corollary}\label{cor:Hf}
Let $f=\lambda w+\mu$ with $\lambda\in\K^{\times}\setminus\{1\}$ and $\mu\in\K$, and put
$s=w+\mu(\lambda-1)^{-1}\in\mathcal H(f)$.
\begin{enumerate}[\rm (i)]
\item The assignment
\begin{align*}
x\longmapsto(\lambda-1)^{-1}x,\qquad y\longmapsto y,\qquad z\longmapsto s-xy
\end{align*}
defines an isomorphism $\mathcal H_{1,\lambda}\to\mathcal H(f)$.
\item $\mathcal H(f)$ is PI if and only if $\lambda$ is a root of unity. In that
case, if $\lambda$ has order $n\geq2$,
\begin{align*}
Z(\mathcal H(f))=\K\bigl[w-xy,\ x^{n},\ y^{n}\bigr]
\end{align*}
is a polynomial ring in three variables, and $\Oz(\mathcal H(f))\cong\mu_n$ is
generated by conjugation by the normal element $s$.
\item $\mathcal H(f)$ is not Calabi--Yau. Consequently, a noetherian generalized
Heisenberg algebra $\mathcal H(f)$ is Calabi--Yau if and only if $f=w+\mu$ for some
$\mu\in\K$.
\end{enumerate}
\end{corollary}

\begin{proof}
(i) From \eqref{eq:Hf}, $sx=wx+\mu(\lambda-1)^{-1}x=\lambda xw+\mu x+\mu(\lambda-1)^{-1}x=\lambda xs$,
and in the same way $ys=\lambda sy$; moreover $yx-xy=(\lambda-1)s$. So
$\mathcal H(f)$ is generated by $x,y,s$ with $sx=\lambda xs$, $ys=\lambda sy$ and
$yx-xy=(\lambda-1)s$, which is $\mathcal H(\lambda s)$ in the variable $s$. The
isomorphism $\mathcal H_{1,\lambda}\cong\mathcal H(\lambda s)$ is the one written down
in the proof of \cite[Theorem~2.9]{LopesSuarezSuarez2025}; for the reader's
convenience, $s-xy$ commutes with $x$ because
\begin{align*}
(s-xy)x=\lambda xs-x\bigl(xy+(\lambda-1)s\bigr)=x(s-xy),
\end{align*}
symmetrically with $y$, and
$y\cdot(\lambda-1)^{-1}x-\lambda(\lambda-1)^{-1}x\cdot y=s-xy$, so the relations
\eqref{eq:rel} with $p=1$, $q=\lambda$ hold for the images.

(ii) $\mathcal H_{1,\lambda}$ is PI exactly when $\lambda$ is a root of unity
\cite[Proposition~3.11]{Gaddis2016}, since $pq=\lambda\neq1$. For $p=1$ and
$q=\lambda$ of order $n$ we have $m=1$, $\ell=\ord(pq)=n$ and $e=1$.
Proposition~\ref{prop:regular} gives $Z(\mathcal H_{1,\lambda})=\K[z,x^{n},y^{n}]$, whose
image under (i) is $\K[s-xy,x^{n},y^{n}]=\K[w-xy,x^{n},y^{n}]$.
Corollary~\ref{ozono2H_pq} gives $\Oz\cong\mu_n\times\mu_1$, and by
Corollary~\ref{cor:generadores-ozono} it is generated by $\eta_\theta$, because
$\eta_z=\varphi_{1,1}$ and $\eta_{x^{n}}=\varphi_{1,\lambda^{n}}$ are trivial. Finally
$\theta=yx-xy$ is sent to $(\lambda-1)^{-1}(yx-xy)=s$.

(iii) By Corollary~\ref{calabiyau}, $\mathcal H_{p,q}$ is Calabi--Yau only if
$p=q$, and here $p=1\neq\lambda=q$. For the last claim, a noetherian
$\mathcal H(f)$ has $\deg f=1$; if the leading coefficient of $f$ is $1$, then
$\mathcal H(f)$ is Calabi--Yau by \cite[Corollary~3.5]{LopesSuarezSuarez2025}, and
otherwise it is not, by what we just proved.
\end{proof}

\begin{remark}
In (iii) we do not need $\lambda$ to be a root of unity: the Nakayama
automorphism of Proposition~\ref{prop:CY} is computed for all $p,q$. For
$\mu=0$ the isomorphism of (i) is graded once $w$ is given degree $2$, and then
$\mathcal H(\lambda w)$ is the graded down-up algebra $A(1+\lambda,-\lambda)$, with
characteristic roots $1$ and $\lambda$. These are the algebras of
Proposition~\ref{prop:LND} and Corollary~\ref{cor:LND-downup} below; for $\lambda=-1$ one recovers $A(0,1)$.
\end{remark}

\begin{remark}\label{rem:centro-no-fijo}
A regular center might look odd next to the invariant theory of down-up
algebras: a down-up algebra has no quasi-reflections of finite order, so no
fixed ring $A^{H}$ with $H\neq1$ finite is AS regular
\cite[Proposition~6.4]{KirkmanKuzmanovichZhang2008Rigidity}, see also
\cite[Proposition~0.2]{KirkmanKuzmanovichZhang2015DownUp}; Gaddis and Yee prove
the same kind of rigidity for the graded algebras $B_q(t^{d})$ with $q\neq\pm1$
and $d\geq2$ \cite[Theorem~4.9]{GaddisYee2026TrivialOzone}. There is no conflict,
because the center is not the fixed ring of any group of automorphisms of
$\mathcal H_{p,q}$, finite or not. If $Z(\mathcal H_{p,q})=\mathcal H_{p,q}^{H}$,
then $H$ fixes the center pointwise, so $H\subseteq\Oz(\mathcal H_{p,q})$, and
then
\begin{align*}
\mathcal H_{p,q}^{H}\supseteq\mathcal H_{p,q}^{\Oz(\mathcal H_{p,q})}
\supsetneq Z(\mathcal H_{p,q})
\end{align*}
by Corollary~\ref{cor:noskew}. For the same reason the results of
\cite{KirkmanKuzmanovichZhang2015DownUp} on fixed rings say nothing directly
about the center, and in the next subsection we work with the classical
invariants of $\mu_\ell$ acting on $\K[z,\theta]$ instead. Since the center is
commutative, the various notions of complete intersection proposed in
\cite{KirkmanKuzmanovichZhang2015NCI} all agree for it
\cite[Lemma~1.8]{KirkmanKuzmanovichZhang2015NCI}, so the phrase has only one
meaning below.
\end{remark}

\subsection{The Gorenstein property}\label{sec:gorenstein}

The invariant-theoretic half of Proposition~\ref{prop:presentacion} turns the
Gorenstein property of the center into an arithmetic condition on $p$ and $q$.
Write $P=\K[z,\theta]$ and $R=P^{\mu_\ell}$ for the action \eqref{eq:accion},
and put
\begin{align}
d_A=\gcd(A,\ell),\qquad d_B=\gcd(B,\ell),\qquad h=\frac{\ell}{d_Ad_B}.
\label{eq:h}
\end{align}
The action is faithful, so $\gcd(A,B,\ell)=1$ and hence $\gcd(d_A,d_B)=1$; in
particular $d_Ad_B$ divides $\ell$ and $h$ is an integer.

\begin{lemma}\label{lem:gcd-gor}
In the notation of \eqref{eq:h},
\begin{align*}
h=\gcd(m,n),\qquad \frac{m}{h}=d_B,\qquad \frac{n}{h}=d_A .
\end{align*}
In particular $h$ does not depend on the choice of generator $g$ in
\eqref{eq:g}.
\end{lemma}

\begin{proof}
We have $m=\ell/d_A$ and $n=\ell/d_B$, and $\operatorname{lcm}(m,n)=\ell$
because $\Gamma$ is cyclic of order $\ell$. Since $\gcd(d_A,d_B)=1$,
\begin{align*}
\gcd(m,n)=\frac{mn}{\ell}=\frac{\ell}{d_Ad_B}=h,
\end{align*}
and then $m/h=d_B$ and $n/h=d_A$. The first expression involves only $m$ and
$n$.
\end{proof}

\noindent The condition that governs everything in this subsection is
\begin{align}
p^{m/h}=q^{n/h},\qquad h=\gcd(m,n).
\tag{G}
\label{eq:gor}
\end{align}
It is written in terms of $p$, $q$, $m$ and $n$ alone, so no choice is involved
in stating it.

\begin{proposition}\label{prop:gorenstein-crit}
$R$ is Gorenstein if and only if \eqref{eq:gor} holds. In that case
$Z(\mathcal H_{p,q})$ is Gorenstein.
\end{proposition}

\begin{proof}
Write $A=d_A\bar a$ and $B=d_B\bar b$, so that $\gcd(\bar a,\ell/d_A)=1$ and
$\gcd(\bar b,\ell/d_B)=1$. Let $G=\mu_\ell$ and
\begin{align*}
G_z=\{\zeta\in G:\zeta^{A}=1\},\qquad
G_\theta=\{\zeta\in G:\zeta^{B}=1\},
\end{align*}
of orders $d_A$ and $d_B$. Since $\gcd(A,B,\ell)=1$ we have
$G_z\cap G_\theta=1$. By \eqref{eq:accion} an element $\zeta$ acts on
$\K z\oplus\K\theta$ by $\mathrm{diag}(\zeta^{A},\zeta^{-B})$, so it is a
pseudo-reflection exactly when one of the two eigenvalues equals $1$ and the
other does not; that is, the pseudo-reflections are the nonidentity elements of
$G_z\cup G_\theta$. Hence the subgroup $N$ they generate is $G_zG_\theta$, of
order $d_Ad_B$, and $G/N$ is cyclic of order $h$.

The group $G_z$ fixes $z$ and acts on $\theta$ with weight $-B$; as
$\gcd(d_A,B)=\gcd(d_A,d_B)=1$, its invariants are $\K[z,\theta^{d_A}]$, and
symmetrically for $G_\theta$. Therefore
\begin{align*}
P^{N}=\K[X,Y],\qquad X:=z^{d_B},\quad Y:=\theta^{d_A},
\end{align*}
and $R=P^{G}=(P^{N})^{G/N}=\K[X,Y]^{G/N}$.

Put $w=\zeta_\ell^{d_Ad_B}$, of order $h$. The generator $\zeta_\ell$ of $G$
acts on $X$ by $\zeta_\ell^{Ad_B}=w^{\bar a}$ and on $Y$ by
$\zeta_\ell^{-Bd_A}=w^{-\bar b}$, so $G/N$ acts with weights $\bar a$ and
$-\bar b$. Since $h$ divides $\ell/d_A$ we get $\gcd(\bar a,h)=1$, and
symmetrically $\gcd(\bar b,h)=1$; so no nonidentity element of $G/N$ fixes $X$
or $Y$, and $G/N$ contains no pseudo-reflections. The determinant of the action
of $w$ on $\K X\oplus\K Y$ is $w^{\bar a-\bar b}$, whence
\begin{align*}
G/N\subseteq\mathrm{SL}_{2}(\K)\iff h\mid \bar a-\bar b,
\end{align*}
and Watanabe's criterion \cite{Watanabe1974} gives that
$\K[X,Y]^{G/N}$ is Gorenstein precisely in that case.

It remains to identify that congruence with \eqref{eq:gor}. By
Lemma~\ref{lem:gcd-gor}, $p^{m/h}=q^{n/h}$ reads $p^{d_B}=q^{d_A}$, that is
$Ad_B\equiv Bd_A\pmod\ell$, that is $d_Ad_B(\bar a-\bar b)\equiv0\pmod\ell$,
which is $h\mid\bar a-\bar b$.

For the last claim, Proposition~\ref{prop:presentacion} gives
$Z(\mathcal H_{p,q})\cong R[U,V]/(UV-F)$. If $R$ is Gorenstein then so is
$R[U,V]$, and $UV-F$ is a nonzerodivisor, $Z(\mathcal H_{p,q})$ being a domain.
\end{proof}

\begin{remark}\label{rem:danilov}
By the Danilov--Stanley criterion
\cite[Section~6.3]{BrunsHerzog},
$R=\K[L_{\geq0}]$ is Gorenstein if and only if
\begin{align*}
L\cap\Z_{\geq1}^{2}=c+L_{\geq0}
\end{align*}
for some $c\in L\cap\Z_{\geq1}^{2}$.
Thus the Gorenstein condition in \eqref{eq:gor}
can also be characterized in terms of the
interior lattice points of $L_{\geq0}$.
\end{remark}


\begin{example}\label{ex:gor}
If $p=q$, then $m=n=h$, and condition~\eqref{eq:gor}
holds. Hence $R$ is Gorenstein.

If $p=\zeta_5$ and $q=\zeta_5^{2}$, then $m=n=h=5$,
and condition~\eqref{eq:gor} fails. Thus $R$ is not
Gorenstein. In this case,
\begin{align*}
L=\{(i,j)\in\Z^2:i\equiv 2j\pmod 5\},
\end{align*}
and $(2,1)$ and $(1,3)$ are incomparable minimal
elements of $L\cap\Z_{\geq1}^{2}$.

Finally, let $g=\zeta_{30}$, $p=g^3$ and $q=g^2$,
as in Example~\ref{ex:30}. Then $m=10$, $n=15$,
$h=5$, and
\begin{align*}
p^2=g^6=q^3.
\end{align*}
Hence $R$ is Gorenstein, although $p\neq q$.
\end{example}

\begin{proposition}\label{prop:descenso}
The ring $R$ is Gorenstein if and only if
$Z(\mathcal H_{p,q})$ is Gorenstein. Consequently,
$Z(\mathcal H_{p,q})$ is Gorenstein if and only if
condition~\eqref{eq:gor} holds.
\end{proposition}

\begin{proof}
Set $Z:=Z(\mathcal H_{p,q})$. By
Proposition~\ref{prop:presentacion},
\begin{align*}
Z\cong R[U,V]/(UV-F).
\end{align*}

Suppose first that $R$ is Gorenstein. Then
$R[U,V]$ is Gorenstein. Since $UV-F$ is a
nonzerodivisor in $R[U,V]$, the quotient $Z$
is Gorenstein.

Conversely, suppose that $Z$ is Gorenstein.
By Proposition~\ref{prop:presentacion}, $Z$
is a free $R$-module with basis
\begin{align*}
\{U^k:k\geq0\}\cup\{V^k:k\geq1\}.
\end{align*}
In particular, the extension $R\subseteq Z$
is faithfully flat.

Let $\mathfrak m$ be a maximal ideal of $R$.
By faithful flatness, there exists a prime
ideal $\mathfrak n$ of $Z$ lying over
$\mathfrak m$. The induced homomorphism
\begin{align*}
R_{\mathfrak m}\longrightarrow Z_{\mathfrak n}
\end{align*}
is flat and local. Since $Z$ is Gorenstein,
so is $Z_{\mathfrak n}$. By Gorenstein
descent along flat local homomorphisms
\cite[Proposition~3.3.14]{BrunsHerzog},
$R_{\mathfrak m}$ is Gorenstein.

As $\mathfrak m$ was arbitrary, $R$ is
Gorenstein.
\end{proof}

\begin{corollary}\label{cor:ci}
$Z(\mathcal H_{p,q})$ is a complete intersection if and only if it is
Gorenstein, that is, if and only if \eqref{eq:gor} holds. It is a hypersurface
if and only if $|\mathcal B|+2-\delta\leq4$.
\end{corollary}

\begin{proof}
A complete intersection is Gorenstein. Conversely, if $Z(\mathcal H_{p,q})$ is
Gorenstein then so is $R$, by Proposition~\ref{prop:descenso}. In the proof of
Proposition~\ref{prop:gorenstein-crit} this gives $h\mid\bar a-\bar b$, and the
group $G/N$ acts on $\K[X,Y]$ with weights $\bar a$ and
$-\bar b\equiv-\bar a$, with $\gcd(\bar a,h)=1$; so it is the
cyclic group of order $h$ acting with weights $1$ and $-1$. Its invariants are
spanned by the $X^{i}Y^{j}$ with $i\equiv j\pmod h$, hence
\begin{align*}
R&=\K[X^{h},Y^{h},XY]\cong\K[u,v,w]/(uv-w^{h}),
\end{align*}
and $R=\K[X,Y]$ when $h=1$. By Proposition~\ref{prop:presentacion},
\begin{align*}
Z(\mathcal H_{p,q})&\cong\K[u,v,w,U,V]/\bigl(uv-w^{h},\ UV-\widetilde F\bigr),
\end{align*}
with $\widetilde F$ any lift of $F$. This quotient has dimension three, so the
ideal has height two; being generated by two homogeneous elements of a
polynomial ring, it is generated by a regular sequence.

For the second claim, if $Z(\mathcal H_{p,q})$ needs at most four homogeneous
generators, the kernel of a minimal presentation is a prime of height at most
one in a polynomial ring, hence principal. Conversely, a hypersurface of
dimension three needs at most four generators, and the number of generators is
$|\mathcal B|+2-\delta$ by Corollary~\ref{cor:generadores}.
\end{proof}

\noindent Examples of each step: $(p,q)=(1,\zeta_3)$ gives a regular center;
$(\zeta_3,\zeta_3)$ a hypersurface which is not regular;
$(\zeta_{30}^{3},\zeta_{30}^{2})$ a complete intersection which is not a
hypersurface; and $(\zeta_5,\zeta_5^{2})$ a Cohen--Macaulay center which is not
Gorenstein.

\begin{corollary}\label{cor:no-gorenstein}
$Z(\mathcal H_{\zeta_5,\zeta_5^{2}})$ is not Gorenstein. In particular, the center
of a noetherian PI down-up algebra need not be Gorenstein, which settles the problem left open in \cite[Section~5.1]{ChanGaddisWonZhang2025}.
\end{corollary}

\begin{proof}
Combine Example~\ref{ex:gor} with Proposition~\ref{prop:descenso}; the algebra
$\mathcal H_{\zeta_5,\zeta_5^{2}}$ is a noetherian PI down-up algebra by
Proposition~\ref{prop:Hpq-downup}.
\end{proof}

\noindent This is the smallest possible $\ell$. Let $\ell\leq4$ and $pq\neq1$.
If $\gcd(m,n)=1$, both sides of \eqref{eq:gor} equal $1$. If $m=n\in\{3,4\}$,
then $p\neq q$ would force $q=p^{-1}$, so $p=q$. If $\{m,n\}=\{2,4\}$, the root
of order two is $-1$ and the square of the other is also $-1$.

\begin{remark}
Corollary~\ref{cor:regular-downup} and
Proposition~\ref{prop:descenso}
together answer both halves of
\cite[Question~6.4]{ChanGaddisWonZhang2024OzoneGroupsCenters}: the center is
regular exactly when $\alpha+\beta=1$, and Gorenstein exactly
when $\omega_1^{\,n_1/h}\omega_2^{\,n_2/h}=1$, which is
\eqref{eq:gor} in terms of the characteristic roots. The first also
settles the conjecture stated just before that question, that
$(\alpha,\beta)=(0,1)$ is the only case with $Z(A)$ regular; the answer is
negative, since every root of unity $r\neq1$ gives a noetherian PI
graded down-up algebra $A(1+r,-r)$, with characteristic roots $1$ and $r$,
whose center is regular. Taking $r=\zeta_{3}$ gives
$(\alpha,\beta)=(1+\zeta_{3},-\zeta_{3})$ and
$Z(\mathcal H_{1,\zeta_{3}})=\K[z,x^{3},y^{3}]$, and $(0,1)$ is the
case $r=-1$. By Corollary~\ref{cor:ci} the Gorenstein half of the answer
also settles when the center is a complete intersection, which is the
down-up analogue of
\cite[Question~6.1(1)]{ChanGaddisWonZhang2024OzoneGroupsCenters}.
\end{remark}

\section{The ozone group of \texorpdfstring{$\mathcal H_{p,q}$}{Hpq}}\label{sec:ozono}

We keep the notation of Section~\ref{sec:center}: $p,q$ are roots of unity with
$pq\neq1$, and $m=\ord(p)$, $n=\ord(q)$, $\ell=\mathrm{lcm}(m,n)$,
$t=\ord(pq)$, $e=\ell/t$. By \cite[Corollary~4.7]{ChanGaddisWonZhang2025},
the ozone group of a noetherian PI graded down-up
algebra consists of diagonal automorphisms and is
therefore abelian. We now determine this group
explicitly for $\mathcal H_{p,q}$.

\begin{proposition}\label{diagHpq}
Let $p,q\in\K^{\times}$ be roots of unity with $pq\neq1$. Every
$\varphi\in\Oz(\mathcal H_{p,q})$ is of the form $\varphi_{a,b}$ of
\eqref{eq:phi}, for some $a,b\in\K^{\times}$. In particular
$\Oz(\mathcal H_{p,q})$ is abelian.
\end{proposition}

\begin{proof}
By Proposition~\ref{prop:Hpq-downup} and
\cite[Corollary~4.7]{ChanGaddisWonZhang2025} there are $a,b\in\K$ with
$\varphi(x)=ax$ and $\varphi(y)=by$, and these are nonzero since $\varphi$ is
an automorphism. Then
\begin{align*}
\varphi(z)=\varphi(y)\varphi(x)-q\varphi(x)\varphi(y)=ab(yx-qxy)=ab\,z,
\end{align*}
so $\varphi=\varphi_{a,b}$. Abelianity follows from Lemma~\ref{lem:phi}, or
from the same corollary.
\end{proof}

\begin{theorem}\label{ozonoH_pq}
With $e=\ell/t$,
\begin{align*}
\Oz(\mathcal H_{p,q})=\bigl\{\varphi_{a,b}:a,b\in\mu_\ell,\;(ab)^{e}=1\bigr\}.
\end{align*}
\end{theorem}

\begin{proof}
By Proposition~\ref{diagHpq} every element of $\Oz(\mathcal H_{p,q})$ is of the
form \eqref{eq:phi}, so it remains to decide when $\varphi_{a,b}$ fixes the
center pointwise. The center is $\Z^{2}$-graded, so by \eqref{eq:peso} this
happens if and only if $a^{u}b^{v}=1$ for every $(u,v)$ in its support. Since
$(0,0)\in L_{\geq0}$ we have $0\in E$, and Corollary~\ref{cor:soporte} puts
$(\ell,0)$ and $(0,\ell)$ in the support, giving $a^{\ell}=b^{\ell}=1$. Modulo
these the remaining conditions read $(ab)^{\varepsilon}=1$ for all
$\varepsilon\in E$, that is $(ab)^{\gcd E}=1$; and $\gcd E=e$ by
Lemma~\ref{lem:aritmetica}.
\end{proof}

\begin{corollary} \label{ozono2H_pq}
$$\Oz(\mathcal H_{p,q})\cong\mu_\ell\times\mu_e.$$
\end{corollary}

\begin{proof}
Define $\Psi(\varphi_{a,b})=(a,ab)$. It lands in $\mu_\ell\times\mu_e$ by
Theorem~\ref{ozonoH_pq}, and it is a homomorphism because
$\varphi_{a,b}\varphi_{c,d}=\varphi_{ac,bd}$ by Lemma~\ref{lem:phi}, so that
\begin{align*}
\Psi(\varphi_{a,b}\varphi_{c,d})=\bigl(ac,(ab)(cd)\bigr)
=\Psi(\varphi_{a,b})\Psi(\varphi_{c,d}).
\end{align*}
It is injective, since $a=c$ and $ab=cd$ force $b=d$. For surjectivity, given
$(a,\lambda)\in\mu_\ell\times\mu_e$ put $b:=a^{-1}\lambda$; as $e\mid\ell$ we
have $\lambda\in\mu_\ell$, hence $b\in\mu_\ell$, and $\Psi(\varphi_{a,b})=(a,\lambda)$.
\end{proof}

\begin{corollary}
If \(\gcd(m,n)=1\), then
$\Oz(\mathcal H_{p,q})\cong\mu_{mn}$.
\end{corollary}

\begin{proof}
Since $\gcd(m,n)=1$, the element $pq$ has order $\mathrm{lcm}(m,n)=mn=\ell$,
so $t=\ell$ and $e=1$. By Corollary~\ref{ozono2H_pq},
$\Oz(\mathcal H_{p,q})\cong\mu_{mn}\times\mu_{1}\cong\mu_{mn}$.
\end{proof}

\begin{example}\label{ex:30}
Let $\K=\C$ and let $g=\zeta_{30}=e^{2\pi i/30}$. Take $p=g^{3}$ and $q=g^{2}$,
so $\ord(p)=10$, $\ord(q)=15$ and $\ell=30$; also $pq=g^{5}\neq1$ and
$g=q^{2}p^{-1}$, so $\Gamma=\langle p,q\rangle=\langle g\rangle$. Since
$\ord(pq)=\ord(g^{5})=6$, we get $t=6$ and $e=\ell/t=5$, whence
$$\Oz(\mathcal H_{p,q})\cong\mu_{30}\times\mu_{5}\cong\Z_{30}\times\Z_{5}.$$
For comparison with \cite{BeraMandalNandy2024}: $p^{2}=g^{6}=q^{3}$ has order
$30/\gcd(30,6)=5=\gcd(m,n)$, so one may take $(r,s)=(2,3)$ and
$\Omega=z^{2}\theta^{3}$; indeed $\gcd(10,15,5)=5=e$, as
Lemma~\ref{lem:aritmetica} predicts.
\end{example}

The following result determines the ozone group of the Kirkman--Small quantum Heisenberg algebra $\mathcal H_q$.

\begin{theorem}\label{ozonoh_q}
Let \(q\) be a primitive \(m\)-th root of unity, with \(m\geq3\). Then
$$
\Oz(\mathcal H_q)
\cong
\mu_m\times\mu_{\gcd(m,2)}\cong
\begin{cases}
\mu_m, & m\text{ odd},\\
\mu_m\times\mu_2, & m\text{ even}.
\end{cases}
$$
\end{theorem}

\begin{proof}
Taking \(p=q\), we have \(\mathcal H_{p,q}\cong\mathcal H_q\) by exchanging $x$
and $y$. Since \(m\geq3\),
\(q^2\neq1\) and therefore $pq=q^2\neq1$. Moreover
$\ord(p)=\ord(q)=m$ and $\ell=m$. Since $t=\ord(q^{2})=m/\gcd(m,2)$, we get
$e=\ell/t=\gcd(m,2)$, and Corollary~\ref{ozono2H_pq} gives
$\Oz(\mathcal H_q)\cong\mu_m\times\mu_{\gcd(m,2)}$. Finally, $\gcd(m,2)$ is $1$
for $m$ odd and $2$ for $m$ even.
\end{proof}

\noindent For $m=4$ we have $\mathcal H_{i}\cong A(0,-1)$, and Theorem~\ref{ozonoh_q} and
Corollary~\ref{cor:centro-Hq} recover \cite[Example~4.8(2)]{ChanGaddisWonZhang2025},
where the generators $\eta_{x^{2}}$ and $\eta_{\Omega_1}$ also appear.

\begin{corollary}
Let \(p\) be a primitive \(m\)-th root of unity and let
\(q=p^k\), with \(m\nmid(k+1)\). Then
$$
\Oz(\mathcal H_{p,p^k})
\cong
\mu_m\times\mu_{\gcd(m,k+1)}.
$$
\end{corollary}

\begin{proof}Since $n=\ord(q)=m/\gcd(m,k)$ divides $m$, we have $\ell=m$, and
$m\nmid(k+1)$ gives $pq=p^{k+1}\neq1$. Hence $t=\ord(p^{k+1})=m/\gcd(m,k+1)$ and
$e=\ell/t=\gcd(m,k+1)$; apply Corollary~\ref{ozono2H_pq}.
\end{proof}

\subsection{The order of the ozone group and the fixed ring}

\begin{proposition}\label{ordenozono}
$$\left|\Oz(\mathcal H_{p,q})\right|=\frac{\ell^{2}}{\ord(pq)}.$$
In particular $\Oz(\mathcal H_{p,q})\neq\{\mathrm{id}\}$, since
$|\Oz(\mathcal H_{p,q})|\geq\ell>1$: indeed $\ord(pq)$ divides $|\Gamma|=\ell$, and
$\ell=1$ would force $p=q=1$, hence $pq=1$.
\end{proposition}

\begin{proof}
By Corollary~\ref{ozono2H_pq} we have $|\Oz(\mathcal H_{p,q})|=|\mu_\ell\times\mu_e|=\ell e$,
and $e=\ell/t$ by definition. Readers comparing with
\cite{BeraMandalNandy2024}, where this invariant appears as $\gcd(m,n,r+s)$,
should use the second half of Lemma~\ref{lem:aritmetica}.
\end{proof}

\begin{proposition}\label{prop:fijo}
Let $S=\bigoplus_{e\mid k}\K[z,\theta]_{k}$ be the $e$-th Veronese subalgebra of
$\K[z,\theta]$. Then
\begin{align*}
\mathcal H_{p,q}^{\Oz(\mathcal H_{p,q})}
=\bigoplus_{a\geq0}x^{\ell a}S\ \oplus\ \bigoplus_{b\geq1}y^{\ell b}S
\ \cong\ S[U,V]/(UV-F),
\end{align*}
so $\mathcal H_{p,q}^{\Oz(\mathcal H_{p,q})}$ is commutative, contains $Z(\mathcal H_{p,q})=R[U,V]/(UV-F)$, and
$\rk_{Z}\mathcal H_{p,q}^{\Oz(\mathcal H_{p,q})}=t$.
\end{proposition}

\begin{proof}
By \eqref{eq:peso}, $\varphi_{a,b}$ acts on $(\mathcal H_{p,q})_{(u,v)}$ by
$a^{u}b^{v}$. Writing $\alpha=a$ and $\lambda=ab$, which by
Corollary~\ref{ozono2H_pq} range independently over $\mu_\ell$ and $\mu_e$, this
scalar is $\alpha^{u-v}\lambda^{v}$; it equals $1$ for all $\alpha,\lambda$
exactly when $\ell\mid u-v$ and $e\mid v$. Lemma~\ref{lem:comp} now gives the
displayed basis.

By Lemma~\ref{lem:aritmetica}, $S$ is the ring of invariants of $\K[z,\theta]$
under the subgroup of order $e$ of the $\mu_\ell$-action \eqref{eq:accion}: that
subgroup consists of the $\zeta^{tk}$, which act diagonally by $\zeta^{tkA}$,
and $\gcd(e,A)=1$ makes $\zeta^{tA}$ of order $e$. In particular $R\subseteq S$
and $[\mathrm{Frac}(S):\mathrm{Frac}(R)]=\ell/e=t$. Since $x^{\ell}$ and
$y^{\ell}$ are central and $x^{\ell}y^{\ell}=F$, the argument of
Proposition~\ref{prop:presentacion} applies verbatim and gives
$\mathcal H_{p,q}^{\Oz(\mathcal H_{p,q})}\cong S[U,V]/(UV-F)=S\otimes_{R}Z(\mathcal H_{p,q})$. Hence
$\rk_{Z}\mathcal H_{p,q}^{\Oz(\mathcal H_{p,q})}=\rk_{R}S=t$.
\end{proof}

\begin{corollary}\label{cor:noskew}
$\mathcal H_{p,q}$ is not isomorphic, as a graded algebra generated in degree one, to any PI skew polynomial ring, and
$\mathcal H_{p,q}^{\Oz(\mathcal H_{p,q})}\neq Z(\mathcal H_{p,q})$. More precisely, $\mathcal H_{p,q}$ has rank $\ell^{2}/t$
over $\mathcal H_{p,q}^{\Oz(\mathcal H_{p,q})}$ and $\mathcal H_{p,q}^{\Oz(\mathcal H_{p,q})}$ has rank $t$ over $Z(\mathcal H_{p,q})$.
\end{corollary}

\begin{proof}
By Proposition~\ref{prop:posner} and $\PIdeg(\mathcal H_{p,q})=\ell$ we get
$\rk_{Z}\mathcal H_{p,q}=\ell^{2}$, whereas $|\Oz(\mathcal H_{p,q})|=\ell^{2}/t$ by
Proposition~\ref{ordenozono}. Since $t>1$, the two differ, and by
\cite[Corollary~3.5(3)]{LiuWuZhu2026OzoneGroups} --- a restatement of
\cite[Theorem~F]{ChanGaddisWonZhang2025} --- an algebra generated in degree one
is a skew polynomial ring if and only if $|\Oz|=\rk_Z$, if and only if
$A^{\Oz}=Z$. The hypotheses are met: $\mathcal H_{p,q}$ is noetherian PI AS regular, it is
generated in degree one by $x$ and $y$ since $z=yx-qxy$, and $\Oz(\mathcal H_{p,q})$ is
abelian by Proposition~\ref{diagHpq}; also $t>1$ because $pq\neq1$. For the ranks, Proposition~\ref{prop:fijo} gives
$\rk_{Z}\mathcal H_{p,q}^{\Oz(\mathcal H_{p,q})}=t$, and multiplicativity of ranks over
$\mathrm{Frac}(Z)$ then gives $\rk_{\mathcal H_{p,q}^{\Oz(\mathcal H_{p,q})}}\mathcal H_{p,q}=\ell^{2}/t
=|\Oz(\mathcal H_{p,q})|$.
\end{proof}

\begin{remark}
The first assertion of Corollary~\ref{cor:noskew} also follows from the Hilbert
series: generated in degree one by $x$ and $y$, the algebra $\mathcal H_{p,q}$ has Hilbert
series $\bigl((1-s)^{2}(1-s^{2})\bigr)^{-1}$, which is not $(1-s)^{-N}$ for any
$N$. What the corollary adds is therefore the second assertion, and the two
ranks.
\end{remark}

Chan, Gaddis, Won and Zhang proposed recovering the center as a further fixed
ring of $A^{\Oz(A)}$, and did so for $A(0,\pm1)$
\cite[Section~5.1, Examples~5.4 and~5.5]{ChanGaddisWonZhang2025}. In our family
this always works, with a cyclic group of order $t$.

\begin{proposition}\label{prop:Ct}
Give $z^{i}\theta^{j}$ degree $p^{i}q^{-j}\in\Gamma$. The Veronese ring $S$ of
Proposition~\ref{prop:fijo} is the sum of the components with degree in
$\langle pq\rangle$, a cyclic group of order $t$, and $R$ is the component of
degree $1$. Grading $U$ and $V$ trivially, this makes
$\mathcal H_{p,q}^{\Oz(\mathcal H_{p,q})}\cong S[U,V]/(UV-F)$ into a
$\langle pq\rangle$-graded algebra with
\begin{align*}
\bigl(\mathcal H_{p,q}^{\Oz(\mathcal H_{p,q})}\bigr)_{1}&=Z(\mathcal H_{p,q}).
\end{align*}
Equivalently, the dual group $C_t=\widehat{\langle pq\rangle}$ acts faithfully
on $\mathcal H_{p,q}^{\Oz(\mathcal H_{p,q})}$ with fixed ring
$Z(\mathcal H_{p,q})$, and
\begin{align*}
\rk_{Z(\mathcal H_{p,q})}\mathcal H_{p,q}^{\Oz(\mathcal H_{p,q})}&=|C_t|=t,
&
\rk_{Z(\mathcal H_{p,q})}\mathcal H_{p,q}&=|\Oz(\mathcal H_{p,q})|\cdot|C_t| .
\end{align*}
Moreover $C_t=\Gal\bigl(\mathcal H_{p,q}^{\Oz(\mathcal H_{p,q})}/Z(\mathcal H_{p,q})\bigr)$
in the sense of \cite[Definition~1.1]{ChanGaddisWonZhang2025}.
No nontrivial element of $C_t$ is the restriction of an automorphism of
$\mathcal H_{p,q}$, graded or not.
\end{proposition}

\begin{proof}
In $\Gamma/\langle pq\rangle$ we have $\overline p=\overline q^{-1}$, so
$p^{i}q^{-j}\equiv q^{-(i+j)}$, and $\overline q$ has order $e$ by
Lemma~\ref{lem:aritmetica}. Hence $p^{i}q^{-j}\in\langle pq\rangle$ if and only
if $e\mid i+j$, which describes $S$; and $R$ is the degree-one part by
Proposition~\ref{prop:presentacion}. Since $F\in R$ has degree $1$ and
$S[U,V]/(UV-F)$ is free over $S$ on $\{U^{a}\}_{a\geq0}\cup\{V^{b}\}_{b\geq1}$,
the grading passes to the quotient and its degree-one part is
$R[U,V]/(UV-F)=Z(\mathcal H_{p,q})$. The grading is faithful, because $z^{e}$
and $z^{e-1}\theta$ both lie in $S$ and their degrees differ by $pq$; hence
$\mathrm{Frac}(S)$ has degree $t$ over $\mathrm{Frac}(R)$, and the second rank
follows from Corollary~\ref{cor:noskew}.

For the last claim, let $g$ be an automorphism of $\mathcal H_{p,q}$ that maps
$\mathcal H_{p,q}^{\Oz(\mathcal H_{p,q})}$ to itself and acts there as some
$\chi\in C_t$. The fixed ring of $C_t$ is $Z(\mathcal H_{p,q})$, so $g$ is the
identity on the center. Then $g\in\Oz(\mathcal H_{p,q})$, and $g$ is the identity
on $\mathcal H_{p,q}^{\Oz(\mathcal H_{p,q})}$; hence $\chi=1$. Nothing here uses
the grading of $g$ or the values of $(\alpha,\beta)$.

Finally, each $\chi\in C_t$ acts by an algebra automorphism fixing
$Z(\mathcal H_{p,q})$, so
$C_t\subseteq\Gal\bigl(\mathcal H_{p,q}^{\Oz(\mathcal H_{p,q})}/Z(\mathcal H_{p,q})\bigr)$.
Every element of this Galois group extends uniquely to an automorphism of
$\mathrm{Frac}(\mathcal H_{p,q}^{\Oz(\mathcal H_{p,q})})$ over
$\mathrm{Frac}(Z(\mathcal H_{p,q}))$, a field extension of degree $t$, and
distinct elements have distinct extensions. So the Galois group has at most
$t=|C_t|$ elements, and the inclusion is an equality.
\end{proof}

\subsection{Normal elements}\label{sec:normales}

By Lemma~\ref{lem:19}, the computation of $\Oz(\mathcal H_{p,q})$ reduces to determining the automorphisms induced by conjugation by homogeneous normal elements. We classify these elements and obtain an independent proof of
Theorem~\ref{ozonoH_pq}, without using the explicit description of the center.  The element $z^{m}$ is central
because $p^{m}=1$; the element $x^{\ell}$ commutes with $z$ because
$p^{\ell}=1$, and with $y$ because, by \eqref{eq:k1},
$yx^{\ell}=q^{\ell}x^{\ell}y+\pqnum{\ell}x^{\ell-1}z=x^{\ell}y$, as $t\mid\ell$;
and the same holds for $y^{\ell}$ by Lemma~\ref{lem:anti}. By the PBW basis,
$\mathcal H_{p,q}$ is a finitely generated module over the central subalgebra
$\K[z^{m},x^{\ell},y^{\ell}]$.

\begin{theorem}\label{thm:normales}
Let $p,q\in\K^{\times}$ be roots of unity with $pq\neq1$, and let $\nu\in\mathcal H_{p,q}$ be nonzero, normal and
$\Z^{2}$-homogeneous of bidegree $(u,v)$ with $u\geq v$. Then $t\mid u-v$, and
there is a single coset $C$ of $L$ in $\Z^{2}$ such that
\begin{align*}
\nu=x^{\,u-v}\!\!\sum_{\substack{(i,j)\in C\\ i+j=v}}\!\! c_{ij}\,z^{i}\theta^{j},
\qquad c_{ij}\in\K .
\end{align*}
Conversely, if $t\mid u-v$ then every such element is normal, with
\begin{align*}
\nu x=\xi\,x\nu,\qquad \nu y=\xi^{-1}q^{-(u-v)}\,y\nu,
\qquad \xi=p^{-i}q^{\,j}\ \ \text{for any }(i,j)\in C,
\end{align*}
and $\nu z=p^{\,u-v}z\nu$. The statement for $v\geq u$ is symmetric, with $y^{\,v-u}$ on the right.
\end{theorem}

\begin{proof}
By Lemma~\ref{lem:comp} we may write $\nu=x^{a'}g$ with $a'=u-v$ and
$g=\sum c_{ij}z^{i}\theta^{j}$ homogeneous of degree $v$ in the polynomial ring
$\K[z,\theta]$.

Since $\nu$ is normal we have $\nu x\in\mathcal H_{p,q}\nu$, say $\nu x=w\nu$;
comparing bidegrees, $w$ is homogeneous of bidegree $(1,0)$, and by
Lemma~\ref{lem:comp} that component is $\K x$. Hence $\nu x=\xi\,x\nu$ for some
$\xi\in\K^{\times}$, and likewise $\nu y=\xi_2\,y\nu$ with $\xi_2\in\K^{\times}$. From
$z^{i}\theta^{j}x=p^{-i}q^{j}\,x\,z^{i}\theta^{j}$ we get
$\nu x=x^{a'+1}\sum c_{ij}p^{-i}q^{j}z^{i}\theta^{j}$ and $x\nu=x^{a'+1}g$, so
$p^{-i}q^{j}=\xi$ for every $(i,j)$ in the support of $g$; that is, the support
lies in one coset $C$ of $L$.

On that coset $z^{i}\theta^{j}y=p^{i}q^{-j}\,y\,z^{i}\theta^{j}
=\xi^{-1}y\,z^{i}\theta^{j}$, so $gy=\xi^{-1}yg$ and $\nu y=\xi^{-1}x^{a'}yg$.
On the other hand \eqref{eq:k1} gives $yx^{a'}=q^{a'}x^{a'}y+\pqnum{a'}x^{a'-1}z$,
so for $a'\geq1$, using $x^{a'}y=(q-p^{-1})^{-1}x^{a'-1}(\theta-z)$,
\begin{align*}
\nu y-\xi_2\,y\nu
=x^{a'-1}\Bigl[\tfrac{\xi^{-1}-\xi_2q^{a'}}{q-p^{-1}}(\theta-z)
-\xi_2\pqnum{a'}\,z\Bigr]g .
\end{align*}
Since $\mathcal H_{p,q}$ is a domain and $g\neq0$, this vanishes exactly when the bracket
does. Its $\theta$-coefficient gives $\xi_2=\xi^{-1}q^{-a'}$, and then its
$z$-coefficient reduces to $-\xi_2\pqnum{a'}$, so $\pqnum{a'}=0$, that is
$t\mid a'$ by \eqref{eq:pq}. For $a'=0$ there is nothing to prove.
The converse follows by reversing the preceding argument.

For the last relation, $z$ commutes with $z^{i}\theta^{j}$ and
$zx^{a'}=p^{-a'}x^{a'}z$, so $\nu z=p^{a'}z\nu$. The case $v\geq u$ follows from
Lemma~\ref{lem:anti}.
\end{proof}

\begin{corollary}\label{cor:ozono-normales}
$\Oz(\mathcal H_{p,q})=\{\varphi_{\alpha,\beta}:\alpha\in\mu_\ell,\ \alpha\beta\in\mu_{\ell/t}\}
\cong\mu_\ell\times\mu_{\ell/t}$.
\end{corollary}

\begin{proof}
By Lemma~\ref{lem:19} every element of $\Oz(\mathcal H_{p,q})$ is
$\eta_\nu$ for some homogeneous normal $\nu$, and conversely every such $\eta_\nu$ fixes the center pointwise, since central elements
commute with $\nu$. Suppose first that $\nu$ has bidegree $(u,v)$ with
$u\geq v$. By Theorem~\ref{thm:normales}, $\nu x=\xi x\nu$ and
$\nu y=\xi^{-1}q^{-a'}y\nu$ with $t\mid a'$, hence
$\eta_\nu(x)=\nu^{-1}x\nu=\xi^{-1}x$ and $\eta_\nu(y)=\xi q^{a'}y$; that is,
$\eta_\nu=\varphi_{\alpha,\beta}$ with $\alpha=\xi^{-1}$ and
$\alpha\beta=q^{a'}$. If instead $v>u$, apply $\tau$: the element $\tau(\nu)$ is normal of bidegree $(v,u)$, and reading Theorem~\ref{thm:normales} backwards through
$\tau$ gives $\eta_\nu(x)=\xi'^{-1}q^{-a'}x$ and $\eta_\nu(y)=\xi'y$ with
$t\mid a'$, so again $\alpha\beta\in\langle q^{t}\rangle$. Here $\xi=p^{-i}q^{j}$ with $i,j\geq0$ arbitrary, so
$\alpha$ runs over all of $\Gamma=\mu_\ell$; and $a'$ runs over $t\Z_{\geq0}$,
so $q^{a'}$ runs over $\langle q^{t}\rangle$, a group of order
$n/\gcd(n,t)=\mathrm{lcm}(n,t)/t=\ell/t=e$. The two choices are independent:
given $\alpha\in\mu_\ell$ and $a'\in t\Z_{\geq0}$, pick $i,j\geq0$ with
$p^{-i}q^{j}=\alpha^{-1}$ and take $\nu=x^{a'}z^{i}\theta^{j}$.
\end{proof}

\begin{corollary}\label{cor:generadores-ozono}
$\Oz(\mathcal H_{p,q})$ is generated by conjugation by the normal elements $z$,
$\theta$ and $x^{t}$. Explicitly,
\begin{align*}
\eta_z=\varphi_{p,p^{-1}},\qquad
\eta_\theta=\varphi_{q^{-1},q},\qquad
\eta_{x^{t}}=\varphi_{1,q^{t}},
\end{align*}
and under the isomorphism $\Psi$ of Corollary~\ref{ozono2H_pq} the subgroup
$\langle\eta_z,\eta_\theta\rangle$ is the factor $\mu_\ell$ while
$\langle\eta_{x^{t}}\rangle$ is the factor $\mu_{\ell/t}$.
\end{corollary}

\begin{proof}
Apply Theorem~\ref{thm:normales} to $\nu=z$, $\theta$ and $x^{t}$, whose
bidegrees are $(1,1)$, $(1,1)$ and $(t,0)$, and read off $\eta_\nu$ as in the
proof of Corollary~\ref{cor:ozono-normales}: for $z$ one gets $\xi=p^{-1}$ and
$a'=0$, for $\theta$ one gets $\xi=q$ and $a'=0$, and for $x^{t}$ one gets
$\xi=1$ and $a'=t$. Hence $\Psi(\eta_z)=(p,1)$, $\Psi(\eta_\theta)=(q^{-1},1)$
and $\Psi(\eta_{x^{t}})=(1,q^{t})$. Now $\langle p,q\rangle=\mu_\ell$ and
$\langle q^{t}\rangle$ has order $\mathrm{lcm}(n,t)/t=\ell/t$ by
Lemma~\ref{lem:aritmetica}, so the three images generate
$\mu_\ell\times\mu_{\ell/t}$.
\end{proof}

\begin{remark}
Two features of the proof of Corollary~\ref{cor:ozono-normales} are worth
pointing out. It never uses the center, so it is independent of any description
of $Z(\mathcal H_{p,q})$. And the element $x^{t}$ is normal precisely because
$\pqnum{t}=0$, which by \eqref{eq:pq} is the same vanishing that produces the
factor $\ell/t$; so the index $t$ of Proposition~\ref{ordenozono} has a direct
meaning. The proof also gives Proposition~\ref{diagHpq} independently of
\cite[Corollary~4.7]{ChanGaddisWonZhang2025}, since every $\eta_\nu$ there is
diagonal on $x$ and $y$.
\end{remark}

\begin{remark}[Unique factorization]\label{rem:UFR}
Liu and Wu proved that every noetherian PI Artin--Schelter regular algebra of
global dimension three is a noetherian unique factorization ring in the sense of
Chatters and Jordan \cite{ChattersJordan1986}, \cite[Theorem~4.7]{LiuWu2026UFR},
and that for connected graded ones the central Picard group is the ozone group
\cite[Proposition~3.3]{LiuWu2026UFR}. For down-up algebras the first fact is
also in \cite[Corollary~7.1]{LaunoisLopes2013}, where it is shown moreover that
$A(\alpha,\beta,0)$ is a noetherian UFD in the sense of Chatters
\cite{Chatters1984} exactly when the group generated by its characteristic
roots is torsion-free; for $\mathcal H_{p,q}$ that group is $\Gamma\cong\mu_\ell$,
so $\mathcal H_{p,q}$ is a UFR but not a UFD.

The classification of normal elements makes this explicit. Let $\mathrm{CCl}$ be
the central class group, that is, reflexive fractional ideals modulo those of the
form $\mathcal H_{p,q}k$ with $0\neq k$ in $Z(Q)=\mathrm{Frac}(Z)$, where $Q$ is
the quotient division ring (Proposition~\ref{prop:posner}; see also
\cite{LV20b}). Since $\mathcal H_{p,q}$ is a connected graded UFR and every element
of the ozone group is conjugation by a regular normal element
(Lemma~\ref{lem:19}), sending a normal element $\nu$ to $\eta_\nu$ identifies
$\mathrm{CCl}$ with $\Oz$, and Corollaries~\ref{ozono2H_pq}
and~\ref{cor:generadores-ozono} give
\begin{align*}
\mathrm{CCl}(\mathcal H_{p,q})\cong\operatorname{Picent}(\mathcal H_{p,q})
\cong\Oz(\mathcal H_{p,q})\cong\mu_\ell\times\mu_{\ell/t},
\end{align*}
generated by the classes of $z\mathcal H_{p,q}$, $\theta\mathcal H_{p,q}$ and
$x^{t}\mathcal H_{p,q}$. The last one witnesses the failure of the UFD property.
The element $x^{t}$ is normal, and by Theorem~\ref{thm:normales} it is not a
product of two normal non-units, so $x^{t}\mathcal H_{p,q}$ is a height-one prime
\cite[Proposition~4.10]{LiuWu2026UFR}; it is not completely prime, because
$x\cdot x^{t-1}$ lies in it while $x$ and $x^{t-1}$ do not, $t$ being at least
$2$. Compare \cite[Lemma~6.5]{LaunoisLopes2013}.
\end{remark}

\subsection{The Calabi--Yau property}\label{sec:CY}
It was shown in \cite[Corollary~2.5]{LopesSuarezSuarez2025} that
$\mathcal H_{p,q}$ is graded skew Calabi--Yau of dimension three.
\begin{proposition}\label{prop:CY}
The Nakayama automorphism of $\mathcal H_{p,q}$ is given by
\begin{align}
\mu_{\mathcal H_{p,q}}(x)=pq^{-1}x,
\qquad
\mu_{\mathcal H_{p,q}}(y)=p^{-1}qy,
\qquad
\mu_{\mathcal H_{p,q}}(z)=z.
\label{eq:naka}
\end{align}
\end{proposition}

\begin{proof}
Using the identification of \(\mathcal H_{p,q}\) as a down-up algebra
established above, with parameter \(\beta=-p^{-1}q\neq0\), and by the computation of the
Nakayama automorphism for graded down-up algebras given in
\cite[(E1.5.6)]{LuMaoZhang2017}, if \(x\) and \(y\) denote the down and up
generators, then \(\mu(y)=-\beta y\) and \(\mu(x)=-\beta^{-1}x\),
so
$$
\mu_{\mathcal H_{p,q}}(y)=-\beta y=p^{-1}qy,
\qquad
\mu_{\mathcal H_{p,q}}(x)=-\beta^{-1}x=pq^{-1}x,
$$
and finally
$\mu_{\mathcal H_{p,q}}(z)=\mu(y)\mu(x)-q\,\mu(x)\mu(y)=z$.
\end{proof}

\begin{corollary}\label{calabiyau}
$\mathcal H_{p,q}$ is Calabi--Yau if and only if $p=q$.
\end{corollary}

\begin{proof}
An algebra \(A\) is Calabi--Yau if and only if \(A\) is skew Calabi--Yau and its
Nakayama automorphism \(\mu_A\) is inner
\cite[Definition~1.1(c)]{LuMaoZhang2017}. Since \(\mathcal H_{p,q}\) is a
connected \(\mathbb N\)-graded domain, its units are scalars and its inner
automorphisms are trivial, so by \eqref{eq:naka} the algebra
\(\mathcal H_{p,q}\) is Calabi--Yau if and only if $p=q$.
\end{proof}

\begin{remark}
The ``if'' part of Corollary~\ref{calabiyau} was proved independently in
\cite[Corollary~3.4]{LopesSuarezSuarez2025}: there $\mathcal H_q$ is written as the
Jacobian algebra of a potential and \cite[Theorem~3.3]{LopesSuarezSuarez2025}
shows that the Calabi--Yau property survives the nonhomogeneous deformations
$\mathcal H_q(f',g')$ of Lopes and Razavinia \cite{LopesRazavinia2022Representations}, with $f'=q^{-1}w+k$ and
$g'=cw^{2}+dw+e$ in $\K[w]$, defined by $wx=xf'$, $yw=f'y$ and $yx-qxy=g'$. Our contribution is the converse, which
needs the Nakayama automorphism and not only a potential.
\end{remark}

\begin{remark}
Since $\Oz(\mathcal H_{p,q})$ is never trivial (Proposition~\ref{ordenozono}),
every $\mathcal H_{q,q}$ is Calabi--Yau with nontrivial ozone group. That the
implication of \cite[Theorem~1.11]{ChanGaddisWonZhang2025} is not an
equivalence was already visible there, in $A(0,-1)\cong\mathcal H_{i,i}$
\cite[Example~4.8(2)]{ChanGaddisWonZhang2025} and in the Calabi--Yau skew
polynomial rings \cite[Lemma~2.2]{ChanGaddisWonZhang2025}.

The algebras $B_q(f)$ of Gaddis
and Yee sit on the other side: they are Calabi--Yau and, for the parameters of
\cite[Theorem~3.13]{GaddisYee2026TrivialOzone}, have trivial ozone group, as the
implication predicts. So in dimension three, among Calabi--Yau PI algebras, both
behaviours occur, and the size of the ozone group is not decided by the
Calabi--Yau property.
\end{remark}

\begin{remark}
For $p=q$ the Gorenstein property of $Z(\mathcal H_p)$ was obtained in
Example~\ref{ex:gor} from the congruence of
Proposition~\ref{prop:gorenstein-crit}. It also follows from
Corollary~\ref{calabiyau} together with
\cite[Theorem~1.11]{ChanGaddisWonZhang2025}, whose hypotheses are satisfied.
Taken with Corollary~\ref{cor:no-gorenstein}, this says that
\cite[Section~5.1]{ChanGaddisWonZhang2025} has a positive answer along the
diagonal $p=q$ and a negative one in general. More precisely, when $m=n$ the
condition \eqref{eq:gor} reduces to $p=q$: for characteristic roots of the same
order, the center is Gorenstein exactly when the algebra is Calabi--Yau.
\end{remark}

\section{The case \texorpdfstring{$pq=1$}{pq=1}}\label{sec:pq1}

Throughout this section, we assume $\K=\C$, since
the description of the center relies on
\cite{Zhao1999}. We now consider the case $pq=1$, which was
excluded from the preceding sections. In this
case, $\theta=z$, and Lemma~\ref{lem:comp}
no longer applies. Moreover,
\begin{align*}
\pqnum{k}=kq^{k-1}\neq0
\end{align*}
for every $k\geq1$, so no positive power of
$x$ is normal.

As shown below, $\mathcal H_{p,q}$ is not
module-finite over its center. Consequently,
Lemma~\ref{lem:19} does not apply, and the
ozone group must be determined by a different
argument. In terms of graded down-up algebras, the
condition $pq=1$ corresponds to coincident
characteristic roots.

\begin{proposition}\label{prop:pq1}
Let $p,q\in\K^{\times}$ be roots of unity with $pq=1$, and let $m=\ord(q)$. Then
\begin{align*}
Z(\mathcal H_{p,q})=\K[z^{m}],
\end{align*}
$\mathcal H_{p,q}$ is not finitely generated as a module over its center, and
\begin{align*}
\{\varphi_{a,b}:a,b\in\K^{\times},\ (ab)^{m}=1\}\subseteq\Oz(\mathcal H_{p,q}).
\end{align*}
In particular $\Oz(\mathcal H_{p,q})$ is infinite for every root of unity $p$,
containing the one-dimensional torus $\{\varphi_{a,a^{-1}}:a\in\K^{\times}\}$.
\end{proposition}

\begin{proof}
With $p=q^{-1}$ the down-up parameters are $\alpha=p^{-1}+q=2q$ and
$\beta=-p^{-1}q=-q^{2}$, so $\alpha^{2}+4\beta=0$, that is $\Delta=0$, and
$\alpha/2=q$ has order $m$. Since $\gamma=0$,
\cite[Theorem~1.3(a)]{Zhao1999} gives
\begin{align*}
Z(\mathcal H_{p,q})=\K\bigl[(du-(\alpha/2)ud)^{m}\bigr]=\K[(yx-qxy)^{m}]=\K[z^{m}].
\end{align*}
(The case $m=1$ is $p=q=1$, where $\mathcal H_{p,q}=U(\mathfrak h)$ and
$Z=\K[z]$; it is allowed in \cite[Theorem~1.3(a)]{Zhao1999} because
$\gamma=0$.) Since $\GKdim\mathcal H_{p,q}=3$ while $\GKdim\K[z^{m}]=1$, the
algebra cannot be a finite module over its center.

By Lemma~\ref{lem:phi}, which places no restriction on $p$ and $q$, each
$\varphi_{a,b}$ is an automorphism, and $\varphi_{a,b}(z^{m})=(ab)^{m}z^{m}$.
So $\varphi_{a,b}$ fixes $\K[z^{m}]$ pointwise exactly when $(ab)^{m}=1$, and
$b=a^{-1}$ always works.
\end{proof}

\begin{remark}\label{rem:Uh}
For $p=q=1$ the inclusion in Proposition~\ref{prop:pq1} is strict, and the
ozone group is not even abelian. In $U(\mathfrak h)$ the assignment
\begin{align*}
\psi_\lambda(x)&=x+\lambda y, & \psi_\lambda(y)&=y, & \psi_\lambda(z)&=z
\end{align*}
is an automorphism, since $[x+\lambda y,y]=[x,y]$, and it fixes $Z=\K[z]$. For
$ab=1$ one gets $\varphi_{a,b}\psi_\lambda(x)=ax+\lambda b\,y$ and
$\psi_\lambda\varphi_{a,b}(x)=ax+\lambda a\,y$, which differ as soon as
$a\neq b$ and $\lambda\neq0$. Compare with the connected graded noetherian PI
AS regular case, where the ozone group is abelian
\cite[Theorem~3.4]{LiuWuZhu2026OzoneGroups}. Connectedness matters there: the
preprojective algebra of \cite[Example~5.1]{ChanGaddisWonZhang2025} is
noetherian and PI, and its ozone group is not abelian.
\end{remark}

\begin{remark}
For $pq\neq1$, Corollary~\ref{ozono2H_pq} shows that $\Oz(\mathcal H_{p,q})$ is finite and abelian, of order $\ell^2/t$. In contrast, when $pq=1$, the ozone group is infinite, and it is nonabelian if $p=q=1$. Thus the PI hypothesis in \cite[Theorem~3.12]{ChanGaddisWonZhang2025}
is essential for the finiteness of the ozone
group, even among connected graded noetherian
Artin--Schelter regular domains of global
dimension three.

The assumption $\K=\C$ in this section is
needed only to invoke
\cite[Theorem~1.3]{Zhao1999} for the
description of the center. The remaining
arguments hold over any algebraically closed
field $\K$ of characteristic zero.
\end{remark}


\section{Applications to questions of Chan, Gaddis, Won and Zhang}
\label{sec:CGWZ}

Section~6 of \cite{ChanGaddisWonZhang2024OzoneGroupsCenters} ends with a list
of questions on noetherian PI Artin--Schelter regular algebras, and down-up
algebras are proposed there as a test case. Question~6.4 was answered in
Section~\ref{sec:center}. Here we look at three more. Two of them turn out to
be the same question for down-up algebras, and an example of Liu, Wu and Zhu \cite{LiuWuZhu2026OzoneGroups} already settles it; what the family $\mathcal H_{p,q}$ adds is an exact
criterion in terms of $p$ and $q$. The third, on affine automorphisms and
locally nilpotent derivations, has a negative answer for every noetherian PI
graded down-up algebra that satisfies its hypotheses.

\subsection{Homological determinant and the element \texorpdfstring{$\mathfrak b$}{b}}

Following \cite[Section~6]{ChanGaddisWonZhang2024OzoneGroupsCenters}, with
the roles of $x$ and $y$ exchanged to match \eqref{eq:rel}, put
\begin{align*}
\mathfrak b&=\Omega_1\Omega_2, & \Omega_i&=yx-\omega_i\,xy,
\end{align*}
where $\omega_1,\omega_2$ are the characteristic roots of $A(\alpha,\beta)$.
With $\omega_1=q$ and $\omega_2=p^{-1}$ this is
\begin{align*}
\mathfrak b&=z\theta .
\end{align*}
For a diagonal graded automorphism $\sigma$ one has
$\sigma(\mathfrak b)=\mathrm{hdet}(\sigma)\,\mathfrak b$
\cite[(E6.2.1)]{ChanGaddisWonZhang2024OzoneGroupsCenters}. From this the
authors get that $A(\alpha,\beta)$ is Calabi--Yau if and only if
$\mathfrak b\in Z(A)$ \cite[Proposition~6.3]{ChanGaddisWonZhang2024OzoneGroupsCenters},
and that for a finite group $H$ of diagonal automorphisms the following are
equivalent: $H$ has trivial homological determinant, $A^{H}$ is Gorenstein,
$\mathfrak b\in A^{H}$
\cite[Proposition~6.5 and Corollary~6.6]{ChanGaddisWonZhang2024OzoneGroupsCenters}.
They then ask whether $\mathfrak b\in A^{\Oz(A)}$ implies $\mathfrak b\in Z(A)$
\cite[Question~6.7]{ChanGaddisWonZhang2024OzoneGroupsCenters}; the question is
repeated in \cite[Question~5.6]{ChanGaddisWonZhang2025}.

The other question is whether the ozone group has trivial homological
determinant exactly when $A$ is Calabi--Yau
\cite[Question~6.2(2)]{ChanGaddisWonZhang2024OzoneGroupsCenters}, repeated as
\cite[Question~5.2(3)]{ChanGaddisWonZhang2025}. Liu, Wu and
Zhu proved that Calabi--Yau implies trivial homological determinant
\cite[Theorem~3.8]{LiuWuZhu2026OzoneGroups} and showed with $A(0,1)$ that the
converse fails \cite[Example~3.9]{LiuWuZhu2026OzoneGroups}.

For down-up algebras the two questions coincide. The ozone group is a finite
group of diagonal automorphisms \cite[Corollary~4.7]{ChanGaddisWonZhang2025},
so by the results just quoted $\mathfrak b\in A^{\Oz(A)}$ says that
$\Oz(A)$ has trivial homological determinant, and $\mathfrak b\in Z(A)$ says
that $A$ is Calabi--Yau. In particular $A(0,1)$ answers
\cite[Question~6.7]{ChanGaddisWonZhang2024OzoneGroupsCenters} in the negative
as well. Since we know $\Oz(\mathcal H_{p,q})$, we can say exactly where the
two conditions separate.

\begin{proposition}\label{prop:hdet}
Let $p,q$ be roots of unity with $pq\neq1$, and let $e=\ell/t$. For
$\varphi_{a,b}\in\Oz(\mathcal H_{p,q})$,
\begin{align*}
\mathrm{hdet}(\varphi_{a,b})&=(ab)^{2}.
\end{align*}
The following are equivalent:
\begin{enumerate}[\rm (i)]
\item $e\mid 2$;
\item $\Oz(\mathcal H_{p,q})$ has trivial homological determinant;
\item $\mathfrak b\in\mathcal H_{p,q}^{\Oz(\mathcal H_{p,q})}$;
\item $\mathcal H_{p,q}^{\Oz(\mathcal H_{p,q})}$ is Gorenstein.
\end{enumerate}
On the other hand, $\mathfrak b\in Z(\mathcal H_{p,q})$ if and only if $p=q$,
if and only if $\mathcal H_{p,q}$ is Calabi--Yau.
\end{proposition}

\begin{proof}
With $\deg x=\deg y=1$ and $\deg z=2$, the PBW monomial $x^{i}z^{j}y^{k}$ has
bidegree $(i+j,\,j+k)$, so by \eqref{eq:peso} the automorphism $\varphi_{a,b}$
multiplies it by $a^{i}(ab)^{j}b^{k}$. Hence
\begin{align*}
\mathrm{Tr}(\varphi_{a,b},s)
&=\frac{1}{(1-as)(1-bs)(1-ab\,s^{2})}
=-(ab)^{-2}s^{-4}+\text{(lower powers of $s$)},
\end{align*}
expanding in $s^{-1}$. As $\mathcal H_{p,q}$ is AS regular of dimension three,
\cite[Lemma~2.6]{JorgensenZhang2000} identifies the leading term with
$-\mathrm{hdet}(\varphi_{a,b})^{-1}s^{-4}$, so
$\mathrm{hdet}(\varphi_{a,b})=(ab)^{2}$. This is also
$\mathrm{hdet}\,g=(\det g|_{A_1})^{2}$ from
\cite[Theorem~1.5]{KirkmanKuzmanovich2005}.

By Corollary~\ref{ozono2H_pq}, $ab$ runs over all of $\mu_e$ as $\varphi_{a,b}$
runs over $\Oz(\mathcal H_{p,q})$, so (ii) holds exactly when
$\mu_e\subseteq\mu_2$, which is (i). The ozone group is finite and consists of
diagonal automorphisms (Proposition~\ref{diagHpq}), so
(ii)$\Leftrightarrow$(iii)$\Leftrightarrow$(iv) is
\cite[Proposition~6.5]{ChanGaddisWonZhang2024OzoneGroupsCenters} with
$H=\Oz(\mathcal H_{p,q})$; the fixed ring is commutative
(Proposition~\ref{prop:fijo}), so AS Gorenstein and Gorenstein mean the same
thing for it. One can also see (i)$\Leftrightarrow$(iii) directly from
$\varphi_{a,b}(z\theta)=(ab)^{2}z\theta$.

Finally $z\theta\,x=p^{-1}q\,x\,z\theta$ and $z\theta\,y=pq^{-1}\,y\,z\theta$,
so $z\theta$ is central exactly when $p=q$, which is the Calabi--Yau
condition of Corollary~\ref{calabiyau}.
\end{proof}

So for $\mathcal H_{p,q}$ the two conditions separate exactly when $e\mid2$ and
$p\neq q$. The smallest case is $A(0,1)=\mathcal H_{1,-1}$, with $\ell=t=2$.
Another is $p=1$, $q=\zeta_3$: there $\ell=t=3$ and $e=1$, so the ozone group
has trivial homological determinant and $\mathfrak b$ is in the fixed ring,
but $\mathfrak b x=\zeta_3 x\mathfrak b$, so $\mathfrak b$ is not central and
$\mathcal H_{1,\zeta_3}$ is not Calabi--Yau. The next example is the one we
find most instructive, because its center is not even Gorenstein.

\begin{corollary}\label{cor:hdet-noCY}
Let $p=\zeta_5$ and $q=\zeta_5^{2}$. Then $\Oz(\mathcal H_{p,q})\cong\mu_5$ has
trivial homological determinant and $\mathfrak b\in
\mathcal H_{p,q}^{\Oz(\mathcal H_{p,q})}$, but $\mathcal H_{p,q}$ is not
Calabi--Yau and $\mathfrak b\notin Z(\mathcal H_{p,q})$.
\end{corollary}

\begin{proof}
Here $\ell=t=5$, so $e=1$, and $p\neq q$. Apply Proposition~\ref{prop:hdet}.
\end{proof}
Corollary~\ref{cor:hdet-noCY} gives a further negative answer to
\cite[Question~6.2(2)]{ChanGaddisWonZhang2024OzoneGroupsCenters} and
\cite[Question~6.7]{ChanGaddisWonZhang2024OzoneGroupsCenters}, that is, to
\cite[Questions~5.2(3) and~5.6]{ChanGaddisWonZhang2025}, now with
$p\neq1\neq q$.
\subsection{The fixed ring of the ozone group}

\begin{corollary}\label{cor:fijo-ci}
The fixed ring $\mathcal H_{p,q}^{\Oz(\mathcal H_{p,q})}$ is a complete
intersection if and only if it is Gorenstein, if and only if $e\in\{1,2\}$. It
is a hypersurface exactly when $e=1$, that is, when $t=\ell$.
\end{corollary}

\begin{proof}
For $e\geq3$ the fixed ring is not Gorenstein by Proposition~\ref{prop:hdet},
so it is not a complete intersection either. If $e=1$ then $S=\K[z,\theta]$,
and if $e=2$ then $S=\K[z^{2},z\theta,\theta^{2}]$. Proposition~\ref{prop:fijo}
then gives
\begin{align}
\mathcal H_{p,q}^{\Oz(\mathcal H_{p,q})}
&\cong \K[z,\theta,U,V]/(UV-F) && (e=1),
\label{eq:fijo-e1}\\
\mathcal H_{p,q}^{\Oz(\mathcal H_{p,q})}
&\cong \K[a,b,c,U,V]/\bigl(ac-b^{2},\ UV-\widetilde F\bigr) && (e=2),
\label{eq:fijo-e2}
\end{align}
where $a,b,c$ map to $z^{2},z\theta,\theta^{2}$ and $\widetilde F$ is a lift of
$F$. The ring in \eqref{eq:fijo-e2} has dimension three, so the ideal has height
two and its two generators form a regular sequence. In \eqref{eq:fijo-e1} we
have $F\in(z,\theta)^{2}$ because $\ell=t\geq2$, so four generators are needed
and the ring is a hypersurface which is not regular. In \eqref{eq:fijo-e2} we
have $\ell=2t\geq4$, so $F$ has degree at least $4$ in $z,\theta$ and lies in
the square of the irrelevant ideal of $S$; five generators are then needed, and
a ring of dimension three with five minimal generators is not a hypersurface.
For $e=1$ and $e=2$ the ozone group is the group $Q_1$, respectively $Q_2$, of
\cite[Lemma~3.4]{KirkmanKuzmanovichZhang2015DownUp}, with $n=\ell$,
respectively $n=\ell/2$,
and these two rings are
\cite[Examples~8.2 and~8.3]{KirkmanKuzmanovichZhang2015DownUp} (for
$(\alpha,\beta)=(0,\pm1)$ see also
\cite[Examples~5.4 and~5.5]{ChanGaddisWonZhang2025}); what
Proposition~\ref{prop:fijo} adds is the uniform description for every $e$ and
the rank over the center. Since the fixed ring is commutative, the
noncommutative notions of complete intersection of
\cite{KirkmanKuzmanovichZhang2015NCI} all reduce to this one
\cite[Lemma~1.8]{KirkmanKuzmanovichZhang2015NCI}.
\end{proof}

\noindent The ozone group is generated by bireflections in the sense of
\cite{KirkmanKuzmanovichZhang2015DownUp}. For $\varphi_{a,b}\neq1$ the trace
series computed in the proof of Proposition~\ref{prop:hdet} has a pole of order
one at $s=1$ exactly when $a=1$, $b=1$ or $ab=1$. The generators
$\eta_z=\varphi_{p,p^{-1}}$, $\eta_\theta=\varphi_{q^{-1},q}$ and
$\eta_{x^{t}}=\varphi_{1,q^{t}}$ of Corollary~\ref{cor:generadores-ozono} are
all of this kind, so the nontrivial ones are bireflections; compare
\cite[Lemma~7.3]{KirkmanKuzmanovichZhang2015DownUp}. None of them is a
quasi-reflection, in agreement with
\cite[Proposition~6.4]{KirkmanKuzmanovichZhang2008Rigidity}. So for $e\geq3$
we get groups generated by bireflections whose fixed ring is not even
Gorenstein.

The algebra of Corollary~\ref{cor:hdet-noCY} has a non-Gorenstein center
(Corollary~\ref{cor:no-gorenstein}), while its ozone fixed ring is a
hypersurface, since $e=1$. The Gorenstein property can therefore be lost on the
way from the fixed ring of the ozone group down to the center.

\begin{remark}\label{rem:gor-independientes}
By Proposition~\ref{prop:hdet} the fixed ring of the ozone group is Gorenstein
exactly when $e\mid2$, and by Corollary~\ref{cor:ci} the center is Gorenstein
exactly when \eqref{eq:gor} holds. The two conditions are independent, and each
of the four combinations occurs:
\begin{itemize}
\item for $(p,q)=(\zeta_3,\zeta_3)$ both hold;
\item for $(p,q)=(\zeta_6,\zeta_6^{2})$ the center is Gorenstein but the
      fixed ring is not, since $e=3\nmid2$;
\item for $(p,q)=(\zeta_5,\zeta_5^{2})$ the fixed ring is Gorenstein but
      the center is not, since \eqref{eq:gor} fails;
\item for $(p,q)=(\zeta_8,\zeta_8^{3})$ neither holds, since \eqref{eq:gor}
      fails and $e=4\nmid2$.
\end{itemize}
In the first and third cases $e=1$, so by Corollary~\ref{cor:fijo-ci} the fixed
ring is in fact a hypersurface. For skew polynomial rings the center is the
fixed ring of the ozone group, so the two properties coincide. Here they do
not, and Proposition~\ref{prop:Ct} says why: the center is a further fixed
ring, $Z(\mathcal H_{p,q})=\bigl(\mathcal H_{p,q}^{\Oz(\mathcal H_{p,q})}\bigr)^{C_t}$.
\end{remark}

\subsection{Affine automorphisms and locally nilpotent derivations}

The last question has two parts, which in the published version of
\cite{ChanGaddisWonZhang2024OzoneGroupsCenters} both carry the label~(5).
Suppose $A$ is generated in degree one and $Z(A)=\K[c_1,\dots,c_r]$ with
$\deg c_i>1$ for all $i$. The first part asks whether every automorphism of
$A$ is affine, that is, maps $A_1$ into $\K\oplus A_1$. The second asks whether
$A$ is LND-rigid, that is, whether $A$ has no nonzero locally nilpotent
derivation \cite[Question~6.2(5)]{ChanGaddisWonZhang2024OzoneGroupsCenters}.
Both hold for skew polynomial rings. For $\mathcal H_{1,q}$ both fail.

\begin{proposition}\label{prop:LND}
Let $q$ be a primitive $n$-th root of unity with $n\geq2$.
\begin{enumerate}[\rm (i)]
\item $\mathcal H_{1,q}$ is generated in degree one, and
$Z(\mathcal H_{1,q})=\K[z,x^{n},y^{n}]$ is a polynomial ring on generators of
degrees $2$, $n$ and $n$.
\item The assignment
\begin{align*}
\delta(x)&=0, & \delta(y)&=x^{n-1}, & \delta(z)&=(1-q)\,x^{n}
\end{align*}
extends to a nonzero locally nilpotent derivation of $\mathcal H_{1,q}$, so
$\operatorname{ML}(\mathcal H_{1,q})\subseteq\ker\delta\subsetneq
\mathcal H_{1,q}$.
\item For every $k\geq0$, $D_k:=x^{kn}\delta$ is a locally nilpotent
derivation, and for $\lambda\in\K$
\begin{align*}
\exp(\lambda D_k)(x)&=x, &
\exp(\lambda D_k)(y)&=y+\lambda x^{(k+1)n-1}, &
\exp(\lambda D_k)(z)&=z+\lambda(1-q)\,x^{(k+1)n}.
\end{align*}
If $\lambda\neq0$ and $(k,n)\neq(0,2)$, this automorphism is not affine. In
particular $\Aut(\mathcal H_{1,q})$ contains automorphisms of arbitrarily
large degree.
\end{enumerate}
\end{proposition}

\begin{proof}
(i) The algebra is generated by $x$ and $y$ because $z=yx-qxy$, and the center
is Proposition~\ref{prop:regular}.

(ii) For $p=1$ the relations \eqref{eq:rel} read $zx=xz$, $zy=yz$ and
$yx-qxy=z$, so $z$ is central, and $x^{n}$ is central because
$\pqnum{n}=(q^{n}-1)/(q-1)=0$. The map $\delta$, defined on the free algebra,
respects the relations:
\begin{align*}
\delta(z)x+z\delta(x)-\delta(x)z-x\delta(z)
&=(1-q)\bigl(x^{n+1}-x^{n+1}\bigr)=0,\\
\delta(z)y+z\delta(y)-\delta(y)z-y\delta(z)
&=(1-q)\bigl[x^{n},y\bigr]+\bigl[z,x^{n-1}\bigr]=0,\\
\delta(y)x+y\delta(x)-q\delta(x)y-qx\delta(y)-\delta(z)
&=x^{n}-qx^{n}-(1-q)x^{n}=0 .
\end{align*}
Since $\delta^{2}$ kills $x$, $y$ and $z$, the Leibniz rule shows that
$\delta$ is locally nilpotent. It is nonzero, and $y\notin\ker\delta$.

(iii) The element $x^{kn}$ is central and $\delta(x^{kn})=0$, so $D_k$ is a
derivation and $D_k^{m}=x^{kmn}\delta^{m}$; in particular $D_k$ is locally
nilpotent, and $\exp(\lambda D_k)$ is an automorphism with inverse
$\exp(-\lambda D_k)$. As $D_k^{2}$ kills the generators, the displayed
formulas follow from $D_k(y)=x^{(k+1)n-1}$ and
$D_k(z)=(1-q)x^{(k+1)n}$. The image of $y$ has a nonzero component of degree
$(k+1)n-1$, which is at least $2$ unless $k=0$ and $n=2$.
\end{proof}

For $k=0$ and $n=2$ the automorphism $\exp(\lambda\delta)$ sends $y$ to
$y+\lambda x$, so it is graded; this is the only case where the multiplier
$x^{kn}$ is really needed.

\begin{corollary}\label{cor:LND-downup}
Let $A=A(\alpha,\beta)$ be a noetherian PI graded down-up algebra. Then $A$
satisfies the hypotheses of
\cite[Question~6.2(5)]{ChanGaddisWonZhang2024OzoneGroupsCenters} if and only
if $\alpha+\beta=1$, and in that case $A$ has non-affine automorphisms and is
not LND-rigid.
\end{corollary}

\begin{proof}
$A$ is generated in degree one. Its center is a connected graded affine domain,
so it is a polynomial ring exactly when it is regular, which by
Corollary~\ref{cor:regular-downup} means $\alpha+\beta=1$. In that case
$A\cong\mathcal H_{1,q}$ or $A\cong\mathcal H_{p,1}$ by
Proposition~\ref{prop:Hpq-downup} and Proposition~\ref{prop:regular}, and
$\mathcal H_{p,1}\cong\mathcal H_{1,p^{-1}}$ by Remark~\ref{rem:simetria}. The
generators of the center have degrees $2$, $n$, $n$ with $n\geq2$, and
Proposition~\ref{prop:LND} applies.
\end{proof}

\begin{remark}\label{rem:LND}
In \cite[Question~6.2(5)]{ChanGaddisWonZhang2024OzoneGroupsCenters} the
symbol $\mathrm{LND}(A)$ stands for the intersection of the kernels of all
locally nilpotent derivations, and the question is written as
$\mathrm{LND}(A)=\{0\}$. That intersection is the Makar--Limanov invariant
$\operatorname{ML}(A)$, which always contains $\K$. Since the question is
introduced through the notion of LND-rigidity, we have read it as
$\operatorname{ML}(A)=A$, which is how it is stated above.

Corollary~\ref{cor:LND-downup} says that inside the class proposed in
\cite{ChanGaddisWonZhang2024OzoneGroupsCenters} as a source of examples the
answer is never positive. The smallest case is $\mathcal H_{1,-1}\cong A(0,1)$,
the algebra singled out there as having regular center, where
$D_1=x^{2}\delta$ gives $y\mapsto y+\lambda x^{3}$. It is also the algebra
used by Liu, Wu and Zhu for Question~6.2(2).

Since $\mathcal H_{1,q}$ is PI, the cancellation results of
\cite[Theorem~0.2]{TVZ21} for non-PI algebras do not apply to it.
\end{remark}

\section*{Acknowledgments}

The authors would like to thank the members of the Algebra Seminar, especially Professor Claudia Gallego, for their valuable discussions and comments. Preliminary results of this work were presented and discussed during the seminar, which provided a stimulating environment for the development of the ideas and the preparation of this manuscript.

\section*{Declaration of AI Use}

All mathematical results, arguments, and proofs presented
in this article were developed independently by the authors,
who also prepared the original manuscript without AI
assistance. OpenAI's ChatGPT was subsequently used to
improve the English language and mathematical exposition,
and to assist in checking proofs and identifying possible
errors. All suggested revisions were independently
evaluated and verified by the authors, who take full
responsibility for the accuracy and integrity of the
final manuscript.


\bibliographystyle{plain}
\bibliography{biblio}
\bigskip
\noindent
James Yair G\'omez Lozano\\
Grupo de \'Algebra y An\'alisis\\
Universidad Pedag\'ogica y Tecnol\'ogica de Colombia\\
Tunja, Colombia\\
james.gomez@uptc.edu.co\\

\noindent H\'ector Julio Su\'arez Su\'arez\\
Grupo de \'Algebra y An\'alisis\\
Universidad Pedag\'ogica y Tecnol\'ogica de Colombia\\
Tunja, Colombia\\
hector.suarez@uptc.edu.co\\

\noindent Helbert Javier Venegas Ram\'irez\\
Grupo ICAAO\\
Universidad Militar Nueva Granada\\
Bogot\'a, Colombia\\
helbert.venegas@unimilitar.edu.co

\end{document}